\documentclass[11pt]{amsart}
\usepackage{amsmath}
\usepackage{amsfonts,amssymb,amsthm, txfonts, pxfonts,amscd}
\usepackage[stretch=10]{microtype}
\usepackage{hyperref}
\usepackage{soul,color}
\def\struckint{\mathop{%
\def\mathpalette##1##2{\mathchoice{##1\displaystyle##2}%
  {##1\textstyle##2}{##1\scriptstyle##2}{##1\scriptscriptstyle##2}}%
\mathpalette
{\vbox\bgroup\baselineskip0pt\lineskiplimit-1000pt\lineskip-1000pt
\halign\bgroup\hfill$}
{##$\hfill\cr{\intop}\cr\diagup\cr\egroup\egroup}%
}\limits}

\usepackage{palatino}
\newtheorem{proposition}{Proposition}[section]

\newtheorem{theorem}[proposition]{Theorem}
\newtheorem{lemma}[proposition]{Lemma}
\newtheorem{corollary}[proposition]{Corollary}

\newtheorem{hypothesis}[proposition]{Hypothesis}

\theoremstyle{definition}
\newtheorem{definition}[proposition]{Definition}

\newtheorem{note}[proposition]{Note}
\newtheorem{remark}[proposition]{Remark}

\newtheorem{example}[proposition]{Example}

\newenvironment{pf*}[1]{\medskip \noindent {\em #1.} }{\endproof \medskip}

\def\bM{\mathop{\bf M}\nolimits}

\def\stochk{\mathop{\mathcal{S}_k}}
\def\simplexk{\mathop{\Delta_k}\nolimits}

\def\ksimplexk{\mathop{\Delta^k_k}\nolimits}

\def\partitionsnk{\mathop{\mathcal{P}_{[n]:k}}\nolimits}

\def\matrixnk{\mathop{\mathcal{M}_{[n]:k}}\nolimits}

\def\labelednk{\mathop{\mathcal{L}_{[n]:k}}\nolimits}
\def\labeledNk{\mathop{\mathcal{L}_{[\infty]:k}}\nolimits}
\def\klabelednk{\mathop{\mathcal{L}^k_{[n]:k}}\nolimits}

\def\cp{\mathop{\rm CP}\nolimits}
\def\equalinlaw{\mathop{=_{\mathcal{L}}}\nolimits}
\def\comp{\mathop{\rm c}\nolimits}

\newcommand{\xnorm}[1]{ \Vert #1 \Vert }

\newcommand{\zz}[1]{\mathbb #1}

\title[Convergence Rates]{Convergence Rates of Markov Chains on Spaces
of Partitions}

\author{Harry Crane}
\address{Rutgers University \\ Department
 of Statistics and Biostatistics\\
 110 Frelinghuysen Road\\
 Piscataway, NJ 08854}
\email{hcrane@stat.rutgers.edu}

\author{Steven P. Lalley}
\address{Department of Statistics\\
 University of Chicago \\
5734 University Avenure\\
 Chicago, IL 60637}
\email{lalley@galton.uchicago.edu}
\urladdr{www.statistics.uchicago.edu/$\sim$lalley}

\thanks{Second author supported by NSF grant DMS  - 0805755}

\date{\today}

\begin{document}


\maketitle
\begin{abstract}
We study the convergence rate to stationarity for a class of exchangeable
partition-valued Markov chains called cut-and-paste chains.  The law governing the transitions
 of a cut-and-paste chain are determined by products of i.i.d.\ stochastic 
 matrices, which describe the chain induced on the simplex by taking
 asymptotic frequencies.  Using this representation, we establish 
 upper bounds for the mixing times of ergodic cut-and-paste chains,
 and under certain conditions on the distribution of the governing
 random matrices we show that the ``cutoff phenomenon'' holds.
\end{abstract}

\section{Introduction} A Markov chain $\{X_{t} \}_{t= 0,1,2,\dotsc }$
on the space $[k]^{\zz{N}}$ of $k-$colorings of the positive integers
$\zz{N}$ is said to be \emph{exchangeable} if its transition law is
equivariant with respect to finite permutations of $\zz{N}$ (that is,
permutations that fix all but finitely many elements of
$\zz{N}$). Exchangeability does not imply that the Markov chain has
the Feller property (relative to the product topology on
$[k]^{\zz{N}}$), but if a Markov chain is both exchangeable and Feller
then it has a simple \emph{paintbox} representation, as proved by
Crane \cite{Crane2012}. In particular, there exists a sequence
$\{S_{t} \}_{t\geq 1}$ of i.i.d.  $k\times k$ random column-stochastic
matrices (the paintbox sequence) such that conditional on the entire
sequence $\{S_{t} \}_{t\geq 1}$ and on $X_{0},X_{1},\dotsc ,X_{m}$,
the coordinate random variables$\{X^{i}_{m+1} \}_{i\in [n]}$ are
independent, and $X^{i}_{m+1}$ has the multinomial distribution
specified by the $X^{i}_{m}$ column of $S_{m+1}$. Equivalently (see
Proposition~\ref{prop:cond indep} in
section~\ref{section:cut-and-paste}), conditional on the paintbox
sequence, the coordinate sequences $\{X^{i}_{m+1} \}_{m\geq 0}$ are
independent, time-inhomogeneous Markov chains on the state space $[k]$
with one-step transition probability matrices
$S_{1},S_{2},\dotsc$. This implies that for any integer $n\geq 1$ the
restriction $X^{[n]}_{t}$ of $X_{t}$ to the space $[k]^{[n]}$ is
itself a Markov chain. We shall refer to such Markov chains $X_{t}$
and $X^{[n]}_{t}$ as \emph{exchangeable Feller cut-and-paste} chains,
or EFCP chains for short. Under mild hypotheses on 
the paintbox distribution (see the discussion in
section~\ref{section:convergence}) the restrictions of EFCP chains
$X^{[n]}_{t}$ to the finite configuration spaces $[k]^{[n]}$ are
ergodic.  The main results of this paper,
theorems \ref{theorem:A}--\ref{theorem:B}, relate the convergence rates
of these chains to  properties of the paintbox process
$S_{1},S_{2},\dotsc $.

\begin{theorem}\label{theorem:A}
Assume that for some $m\geq 1$ there is positive probability that all
entries of the matrix product $S_{m}S_{m-1}\dotsb S_{1}$ are
nonzero. Then the EFCP chain $X^{[n]}$ is ergodic, and it mixes in
$O(\log n)$ steps. 
\end{theorem}

\begin{theorem}\label{theorem:B}
Assume that the distribution of $S_{1}$ is absolutely continuous
relative to Lebesgue measure on the space of $k\times k$
column-stochastic matrices, with density of class $L^{p}$ for some
$p>1$. Then the associated EFCP chains $X^{[n]}$ exhibit the
\emph{cutoff phenomenon}: there exists a positive constant $\theta$
such that for all sufficiently small $\delta ,\varepsilon >0$ the
(total variation) mixing times satisfy
\begin{equation}\label{eq:cutoff-phenomenon}
	(\theta -\delta )\log n \leq t^{(n)}_{\textsc{mix}}
	(\varepsilon)\leq t^{(n)}_{\textsc{mix}} (1-\varepsilon)\leq
	(\theta +\delta )\log n
\end{equation}
for all sufficiently large $n$.
\end{theorem}

Formal statements of these theorems will be given in due course (see
Theorems \ref{theorem:ub} and \ref{theorem:cutoff} in
section~\ref{section:convergence}), and less stringent hypotheses for
the  $O (\log n)$ convergence rate will be given.
In the special case $k=2$ the results are  related to some
classical results for random walks on the hypercube, e.g.\ the
Ehrenfest chain on $\{0,1\}^n$: see example \ref{ex:hypercube}.

The key to both results is that the relative frequencies of the
different colors are determined by the random matrix products
$S_{t}S_{t-1}\dotsb S_{1}$ (see Proposition~\ref{prop:cond
indep}). The hypotheses of Theorem~\ref{theorem:A} ensure that these
matrix products contract the $k-$simplex to a point at least
exponentially rapidly.  The stronger hypotheses of
Theorem~\ref{theorem:B} prevent the simplex from collapsing at a
faster than exponential rate. 

The paper is organized as follows.  In section~\ref{sec:prelims:tv} we
record some simple and elementary facts about total variation distance, and
in section~\ref{sec:prelims:paintbox} we define cut-and-paste Markov
chains formally and establish the basic relation with the paintbox
sequence (Proposition~\ref{prop:cond indep}).  In section
\ref{section:rm} we discuss the contractivity properties of products
of random stochastic matrices. In section \ref{section:convergence} we
prove the main results concerning ergodicity and mixing rates of
cut-and-paste chains, and in section \ref{section:examples} we discuss
some examples of cut-and-paste chains not covered by our main
theorems. Finally, in section \ref{section:projected chains} we deduce
mixing rate and cutoff for projections of the cut-and-paste chain into
the space of ordinary set partitions.

\section{Preliminaries: Total Variation Distance }\label{sec:prelims:tv}

Since the state spaces of interest in our main results are finite, it
is natural to use the total variation metric to measure the distance
between the law $\mathcal{D}(X_m)$ of the chain $X$ at time $m\geq1$
and its stationary distribution $\pi$.  The \emph{total variation
distance} $\xnorm{\mu -\nu}_{TV}$ between two probability measures
$\mu ,\nu$ on a finite or countable set $\mathcal{X}$ is defined by
\begin{equation}\label{eq:definitionTV}
	\xnorm{\mu -\nu}_{TV}=\frac{1}{2} \sum_{x\in \mathcal{X}} |\mu
	(x)-\nu (x)| =\max_{B\subset \mathcal{X}} (\nu (B)-\mu (B)).
\end{equation}
The maximum is attained at $B^*=\{x\,:\, \nu (x)\geq \mu (x)\}$ and,
since the indicator $\mathbf{1}_{B^*}$ is a function only of the
likelihood ratio $d\nu /d\mu$, the total variation distance
$\xnorm{\mu -\nu}_{TV}$ is the same as the total variation distance
between the $\mu -$ and $\nu -$ distributions of any sufficient
statistic.  In particular, if $Y=Y (x)$ is a random variable such that
$d\nu /d\mu$ is a function of $Y$, then
\begin{equation}\label{eq:sufficientStat}
	\xnorm{\mu -\nu}_{TV}=\frac{1}{2} \sum_{y} |\nu (Y=y)-\mu
	(Y=y)|,
\end{equation}
where the sum is over all possible values of $Y (x)$.

Likelihood ratios  provide a useful means for showing that two
probability measures are close in total variation distance.

\begin{lemma}\label{lemma:lrA}
 Fix $\varepsilon >0$, and define
\[
	B_{\varepsilon}= \left\{x \,:\, \left| \frac{\mu (x)}{\nu (x)}
	-1 \right|>\varepsilon \right\} . 
\]
If $\nu (B_{\varepsilon})<\varepsilon$, then $\xnorm{\mu
-\nu}_{TV}<2\varepsilon$.
\end{lemma}
\begin{proof}
By definition of $B_{\varepsilon}$,
$B_{\varepsilon}^{\comp}:=\{x:|\mu(x)-\nu(x)|\leq\varepsilon\nu(x)\}$
and so $(1-\varepsilon)\nu(x)\leq\mu(x)\leq(1+\varepsilon)\nu(x)$ for
every $x\in B_{\varepsilon}^{\comp}$ and
$$\mu(B_{\varepsilon}^{\comp})\geq(1-\varepsilon)\nu(B_{\varepsilon}^{\comp}).$$
By assumption $\nu(B_{\varepsilon})<\varepsilon$, it follows that
$\mu(B_{\varepsilon}^{\comp})\geq(1-\varepsilon)^2$ and
\begin{eqnarray*}
\xnorm{\mu-\nu}_{TV}&=&\frac{1}{2}\left[\sum_{x\in B_\varepsilon}|\mu(x)-\nu(x)|+\sum_{x\in B_{\varepsilon}^{\comp}}|\mu(x)-\nu(x)|\right]\\
&\leq&\frac{1}{2}\left[\sum_{x\in B_{\varepsilon}}\mu(x)+\sum_{x\in B_{\varepsilon}}\nu(x)+\sum_{x\in B_{\varepsilon}^{\comp}}|\mu(x)-\nu(x)|\right]\\
&<&2\varepsilon.\end{eqnarray*}
\end{proof}

The convergence rates of EFCP chains will  be (in the ergodic
cases) determined by the contractivity properties of products
of random stochastic $k\times k$ matrices on the $(k-1)$-dimensional simplex 
\begin{equation}\label{eq:simplex-def}
\simplexk:=\left\{(s_1,\ldots,s_k)^{T}:s_i\geq0 \quad \text{and} \;\;
\sum_i s_i=1 \right\}.
\end{equation}
We now record some preliminary lemmas about convergence of probability
measures on $\simplexk$ that we will need later.  For each
$n\in\mathbb{N}$ and each element $s\in\simplexk$ define a probability
measure $\varrho_s^n$, the \emph{product multinomial-}$s$
\emph{measure} on $[k]^n$ by
\begin{equation}\label{eq:multi-s}
\varrho_s^n(x):=\prod_{j=1}^n s_{x^j}\quad \text{for} \;\; x=x^1x^2\cdots
x^n\in[k]^n. 
\end{equation}
Observe that the vector $m(x):=(m_1,\ldots,m_k)$ of cell
counts defined by $m_j:=\sum_{i=1}^m1_{j}(x^i)$ is sufficient for the
likelihood ratio $\varrho_s^n(x)/\varrho_{s'}^n(x)$ of any two
product-multinomial measures $\varrho_s^n$ and $\varrho_{s'}^n$. 


\begin{corollary}\label{cor:multiTV}
Fix $\delta,\varepsilon>0$.  If $s_n,s'_n$ are two sequences in
$\simplexk$ such that all coordinates of $s_n,s'_n$ are in the
interval $[\delta,1-\delta]$ for every $n$, and if
$\xnorm{s_{n}-s'_{n}}_{\infty}<n^{-1/2-\varepsilon}$, then
$$\lim_{n\rightarrow\infty}\xnorm{\varrho_{s_n}^n-\varrho_{s'_n}^n}_{TV}=0.$$
\end{corollary}

\begin{proof}
This is a routine consequence of Lemma~\ref{lemma:lrA}, as the
hypotheses ensure that the likelihood ratio
$d\varrho_{s_n}^n/d\varrho_{s'_n}^n$ is uniformly close to $1$ with
probability approaching $1$ as $n \rightarrow \infty$.
\end{proof}

A similar argument can be used to establish the following generalization,
which is needed in the case of partitions with $k\geq 3$ classes.  For
$s_1,\ldots,s_k\in\simplexk$, let
$\varrho_{s_1}^{n_1}\otimes\cdots\otimes\varrho_{s_k}^{n_k}$ denote
the product measure on $[k]^{n_1+\cdots+n_k}$ where the first $n_1$
coordinates are i.i.d.\ multinomial-$s_1$, the next $n_2$ are i.i.d.\
multinomial-$s_2$, and so on.

\begin{corollary}\label{cor:prod-multiTV}
Fix $\delta,\varepsilon>0$. For each $i\in [k]$ let $\{s^{i}_{n}
\}_{n\geq 1}$ and $\{t^{i}_{n} \}_{n\geq 1}$ be sequences in
$\simplexk$ all of whose entries are in the interval
$[\delta,1-\delta]$, and let $K^{i}_{n}$ be sequences of nonnegative
integers such that $\sum_{i}K^{i}_{n}=n$. If
$\sum_{i=1}^k\xnorm{t^{i}_{n}-s^{i}_{n}}_{\infty
}<n^{-1/2-\varepsilon}$, then

\begin{equation*}
	\lim_{n\rightarrow\infty}\left\Vert \varrho_{s_n^1}^{K_n^1}\otimes\cdots\otimes
	\varrho_{s_n^k}^{K_n^k}-\varrho_{t_n^1}^{K_n^1}
	\otimes\cdots\otimes\varrho_{t_n^k}^{K_n^k}\right \Vert_{TV}=0.
\end{equation*}

\end{corollary}


In dealing with probability measures that are defined as mixtures, the
following simple tool for bounding total variation distance is useful.

\begin{lemma}\label{lemma:mix}
Let $\mu ,\nu $ be probability measures on a finite or countable space
$\mathcal{X}$ that are both mixtures with respect to a common mixing
probability measure $\lambda (d \theta)$, that is, such that there are
probability measures $\mu_{\theta}$ and $\nu_{\theta}$ for which
\[
	\mu =\int \mu_{\theta} \,d\lambda (\theta)
	\quad \text{and} \quad 
	\nu =\int \nu_{\theta} \,d\lambda (\theta).
\]
If $\xnorm{\mu_{\theta}-\nu_{\theta}}_{TV}<\varepsilon$ for all
$\theta$ in a set of $\lambda -$probability $>1-\varepsilon$ then 
\[
	\xnorm{\mu -\nu}_{TV}<2\varepsilon .
\]
\end{lemma}

Lower bounds on total variation distance between two probabilities
$\mu ,\nu$ are often easier to establish than upper bounds, because
for this one only need find a particular set $B$ such that $\mu
(B)-\nu (B)$ is large. By \eqref{eq:sufficientStat}, it suffices to
look at sets of the form $B=\{Y\in F \}$, where $Y$ is a sufficient
statistic. The following lemma for product Bernoulli measures
illustrates this strategy.  For $\alpha\in[0,1]$, we write $\nu_{\alpha}^n:\varrho_s^n$,
where $s:=(\alpha,1-\alpha)\in\Delta_2$, to denote the product Bernoulli measure
determined by $\alpha$.

\begin{lemma}\label{lemma:bigBernoulliTV}
Fix $\varepsilon >0$. If $\alpha_{m},\beta_{m}$ are sequences in
$[0,1]$ such that $|\alpha_{m}-\beta_{m}|>m^{-1/2+\varepsilon}$, then 
\[
	\lim_{m \rightarrow
	\infty}\xnorm{\nu^{m}_{\alpha_{m}}-\nu^{m}_{\beta_{m}}}_{TV} =1.
\]
\end{lemma}

\begin{proof}
Without loss of generality, assume that
$\alpha_{m}<\beta_{m}$, and let $\gamma_{m}=
(\alpha_{m}+\beta_{m})/2$. Denote by $S_{m}$ the sum of the coordinate
variables. Then by Chebyshev's inequality,
\begin{align*}
	\lim_{m \rightarrow \infty}\nu^{m}_{\alpha_{m}}\{S_{m}<m\gamma_{m}
	\}&=1 \quad \text{and}\\
	\lim_{m \rightarrow \infty}\nu^{m}_{\beta _{m}}\{S_{m}<m\gamma_{m}
	\}&=0 .
\end{align*}
\end{proof}

\begin{remark}\label{remark:bigBernoulliTV}
Similar results holds for multinomial and product-multinomial
sampling. (A) If $s_{n},s'_{n}\in \simplexk$ are distinct sequences of
probability distributions on $[k]$ such that for some coordinate $i\in
[k]$ the $i$th entries of $s_{n}$ and $s'_{n}$ differ by at least
$n^{-1/2+\varepsilon}$, then
\[
	\lim_{n \rightarrow
	\infty}\xnorm{\varrho_{s_n}^n-\varrho_{s'_n}^n}_{TV} =1.
\]
(B) If $s^{i}_{n},t^{i}_{n}\in \simplexk$ are distinct sequences of
probability distributions on $[k]$ such that for some pair $i,j\in [k]$
the $j$th entries of $s^{i}_{n}$ and $t^{i}_{n}$ differ by  at least
$n^{-1/2+\varepsilon}$, then for any sequences $K^{i}_{n}$ of
nonnegative integers such that $\sum_{i}K^{i}_{n}=n$,
\[
	\lim_{n \rightarrow \infty}
	\lim_{n\rightarrow\infty}\left\Vert \varrho_{s_n^1}^{K_n^1}\otimes\cdots\otimes
	\varrho_{s_n^k}^{K_n^k}-\varrho_{t_n^1}^{K_n^1}
	\otimes\cdots\otimes\varrho_{t_n^k}^{K_n^k}\right \Vert_{TV}=1.
\]
These statements follow directly from Lemma~\ref{lemma:bigBernoulliTV} by
projection on the appropriate coordinate variable.
\end{remark}

\section{Preliminaries: CP chains and Paintbox
Representation}\label{sec:prelims:paintbox} 

\subsection{Labeled and unlabeled partitions}\label{section:labeled
partitions} For $k,n\in\mathbb{N}=\{1,2,\ldots\}$, a \emph{labeled}
$k$-ary partition $L$ of $[n]$ is an \emph{ordered} collection
$L:=(L_1,\ldots,L_k)$ of disjoint subsets for which
$\bigcup_{i=1}^kL_i=[n]$. An \emph{unlabeled} $k$-ary partition of
$[n]$ is an \emph{unordered} collection $L:=\{ L_1,\ldots,L_r\}$,
where $r\leq k$, of nonempty, disjoint subsets whose union is $[n]$.
The set $\labelednk$ of labeled $k$-ary partitions of $[n]$ can be
naturally identified with the set $[k]^{n}$ of $k-$colorings of the
set $[n]$, via the map
\[
	L \mapsto l^1l^2\cdots l^n \quad \text{where}\quad l^i=j\Leftrightarrow i\in L_j.
\]
Thus, the multinomial-$s$ measure $\varrho_s^n$ defined in the previous
section induces a measure on $\labelednk$, which we will also denote
by $\varrho_s^n$. There is an obvious and  natural projection $\Pi_{n}:\labelednk
\rightarrow \partitionsnk$ from the set  $\labelednk$ of labeled
partitions to the set $\partitionsnk$ of unlabeled partitions given by
\begin{equation}
\label{eq:projection}\Pi_n(L):=\{L_1,\ldots,L_k\}\backslash\{\emptyset\}.
\end{equation} 
This mapping coincides with the natural projection
\[
	\Pi_n:[k]^{n} \rightarrow [k]^{n}/\sim 
\] 
where $\sim$ is the equivalence relation $l^1l^2\cdots l^n \sim
l_{*}^1l_{*}^2\cdots l_{*}^n$ if and only if there exists a
permutation $\sigma$ of $[k]$ such that $l_{*}^{i}=\sigma (l^{i})$ for
each $i\in [n]$. Some of the Markov chains on $\labelednk$ considered
below have transition laws invariant under such permutations $\sigma$
of the labels $[k]$, and in such cases the Markov chain projects via
$\Pi_{n}$ to a Markov chain on the state space
$\partitionsnk$. This is discussed further in
section~\ref{section:projected chains} below.

\subsection{Matrix operations on $\labeledNk$}\label{section:matrix operations}
The cut-and-paste Markov chain on $\labelednk$ can be described by a
product of i.i.d.\ random set-valued matrices with a special
structural property.  

\begin{definition}\label{definition:partition-matrices}
For any subset $S\subset \zz{N}$, define a $k-$ary (or $k\times k$) \emph{partition
matrix} over $S$ to be a $k\times k$ matrix $M$ whose entries $M_{ij}$
are subsets of $S$ such that every column $M^{j}$ is a labeled $k-$ary
partition of $S$.  For any two $k-$ary partition matrices $M,M'$,
define the product $M*M'=MM'$ by
\begin{equation}
(M*M')_{ij}\equiv(MM')_{ij}:=\bigcup_{1\leq l\leq
k}({M}_{il}\cap{M}'_{lj}), \;\;\text{for all} \; 1\leq i, j\leq
k.\label{eq:boolean matrix product}\end{equation} 
\end{definition}

We write $\matrixnk$ to denote the space of $k\times k$ partition matrices of $[n]$.
Observe that the matrix product defined by
\eqref{eq:boolean matrix product} makes sense for matrices with
entries in any distributive lattice, provided $\cup ,\cap$ are
replaced by the lattice operations. 

As each column of any $M\in\matrixnk$ is a $k$-ary partition of $[n]$,
the set $\matrixnk$ of $k-$ary partition matrices over $[n]$ can be
identified with $\klabelednk$. Furthermore, a $k-$ary partition
matrix $M$ induces a mapping $M:\labelednk \rightarrow \labelednk$, by
\[
	( ML )_{i}=\bigcup_{j} M_{ij}L_{j}.
\]

\begin{lemma}\label{lemma:partition matrix}
Let $k,n\in\mathbb{N}$.  Then
\begin{itemize}
	\item[(i)] for each $L\in\labelednk$, $ML\in\labelednk$ for all $M\in\matrixnk$;
	\item[(ii)] for any $L,L'\in\labelednk$, there exists $M\in\matrixnk$ such that $ML=L'$;
	\item[(iii)] the pair $(\matrixnk,*)$ is a monoid (i.e.,
	semigroup with identity) for every
	$n\in\mathbb{N}$.
\end{itemize}
\end{lemma}
The proof is elementary and follows mostly from the definition
\eqref{eq:boolean matrix product} (the semigroup identity is the
matrix whose diagonal entries are all $[n]$ and whose off-diagonal
entries are $\emptyset$). We now describe the role of the
semigroup $(\matrixnk,*)$ in describing the transitions of the
cut-and-paste Markov chain. 

\subsection{Cut-and-paste Markov
chains}\label{section:cut-and-paste} Fix $n,k\in\mathbb{N}$, let $\mu$
be a probability measure on $\matrixnk$, and let $\varrho_0$ be a
probability measure on $\labelednk$.  The cut-and-paste Markov
chain $X=(X_m)_{m\geq 0}$ on $\labelednk$ with initial distribution
$\varrho_0$ and directing measure $\mu$ is constructed as follows.  Let
$X_0\sim\varrho_0$ and, independently of $X_0$, let $M_1,M_2,\ldots$
be i.i.d.\ according to $\mu$. Define
\begin{equation}\label{eq:cut-and-paste}X_m=M_mX_{m-1}=M_mM_{m-1}\cdots
M_1 X_0,\quad \text{for} \;\; m\geq1.\end{equation}
We call any Markov chain with the above dynamics a
$\cp_n(\mu;\varrho_0)$ chain, or simply a $\cp_n(\mu)$ chain if the
initial distribution is unspecified. Henceforth we will use the
notation $X^{i}_{m}$ to denote the $i$th coordinate variable in
$X_{m}$ (that is, $X^{i}_{m}$ is the color of the site $i\in [n]$
when $X_{m}$ is viewed as an element of $[k]^{[n]}$). 

Our main results concern the class of  cut-and-paste chains whose directing
measures $\mu=\mu_{\Sigma }$ are mixtures of product multinomial
measures $\mu_{S}$, where $S$ ranges over the set $\ksimplexk$ of
$k\times k$ column-stochastic matrices. For any  $S\in\ksimplexk$, the
product multinomial measure $\mu_{S}$ is defined by
\begin{equation}\label{eq:prod-multinomial}
\mu_S(M):=\prod_{j=1}^k\prod_{i=1}^nS(M^j(i),j)\quad \text{for} \;\;
M\in\matrixnk,
\end{equation}
where $M^{j} (i)=\sum_{r}r\mathbf{1}\{i\in M_{rj} \}$ denotes the
index $r$ of the row such that $i$ is an element of $M_{rj}$.  (In
other words, the columns of $M\sim\mu_S$ are independent labeled
$k-$ary partitions, and in each column $M^{j}$ the elements $i\in [n]$
are independently assigned to rows $r\in [k]$ according to draws from the
multinomial distribution $S^{j}$ determined by the $j$th column of
$S$.)  For any Borel probability measure $\Sigma$ on
$\ksimplexk $, we write $\mu_\Sigma$ to denote the
$\Sigma$-mixture of the measures $\mu_S$ on $\matrixnk$, that is,

\begin{equation}\label{eq:matrix
mixture}\mu_\Sigma(\cdot):=\int_{\ksimplexk}\mu_S(\cdot)\Sigma(dS).\end{equation} 

\noindent Crane \cite{Crane2012} has shown that every exchangeable,
Feller Markov chain on the the space $[k]^{\zz{N}}$ of $k-$colorings
of the positive integers is a cut-and-paste chain with directing
measure of the form \eqref{eq:matrix mixture}, and so henceforth, we
shall refer to such chains as \emph{exchangeable Feller cut-and-paste
chains}, or EFCP chains for short.

An EFCP chain on $[k]^{[n]}$ (or $[k]^{\zz{N}}$) with directing
measure $\mu =\mu_{\Sigma}$ can be constructed in two steps, as
follows. First, choose i.i.d.\ stochastic matrices $S_1,S_2,\ldots$
with law $\Sigma$, all independent of $X_0$; second, given
$X_0,S_1,S_2,\ldots$, let $M_1,M_2,\ldots$ be conditionally
independent $k-$ary partition matrices with laws $M_i\sim\mu_{S_i}$
for each $i=1,2,\ldots$, and define the cut-and-paste chain $X_{m}$ by
equation \eqref{eq:cut-and-paste}. This construction is fundamental to
our arguments, and so henceforth, when considering an EFCP chain with
directing measure $\mu_{\Sigma }$, we shall assume that it is defined
on a probability space together with a paintbox sequence
$S_{1},S_{2},\dotsc$.

 For each $m\in\mathbb{N}$, set
\begin{equation}\label{eq:qmkxk}
	 Q_m:=S_mS_{m-1}\cdots S_1.
\end{equation}
Note that $Q_{m}$ is itself a stochastic matrix. Denote by
$\mathcal{S}$ the $\sigma -$algebra generated by the paintbox sequence
$S_{1},S_{2},\dotsc$. 

\begin{proposition}\label{prop:cond indep}
Given $\mathcal{G}:=\sigma (X_{0})\vee \mathcal{S}$, the $n$ 
coordinate sequences $(X_m^i)_{m\geq 0}$, where
$i\in [n]$, are conditionally independent versions of a
time-inhomogeneous Markov chain on $[k]$ with one-step transition
probability matrices $S_{1},S_{2},\dotsc$. Thus, in particular, for
each $m\geq 1$,
\begin{equation}\label{eq:joint law}
	\mathbf{P} (X^{i}_{m}=x^{i}_{m} \; \text{for each} \; i\in
	[n]\,|\, \mathcal{G})			=\prod_{i=1}^{n}
	Q_m(x_m^i,X_0^i). 
\end{equation}
\end{proposition}

\begin{proof}
We prove that the Markov property holds by induction on $m$. The case
$m=1$ follows directly by \eqref{eq:prod-multinomial}, as this implies
that, conditional on $\mathcal{G}$, the coordinate random variables
$X^{i}_{1}$ are independent, with multinomial marginal conditional
distributions given by the columns of $S_{1}$. Assume, then, that the
assertion is true for some $m\geq  1$.  Let $\mathcal{F}_m$ be the
$\sigma$-algebra generated by $\mathcal{G}$ and the random matrices
$M_{1},M_{2},\dotsc ,M_{m}$.  Since the specification
\eqref{eq:cut-and-paste} expresses $X_m$ as a function of
$X_0,M_1,M_2,\ldots, M_{m}$, the random variables $X^{i}_{t}$, where
$t\leq m$, are measurable with respect to $\mathcal{F}_{m}$. Moreover,
given $\mathcal{G}$ the random matrix $M_{m+1}$ is conditionally
independent of $\mathcal{F}_{m}$, with conditional distribution
\eqref{eq:prod-multinomial} where $S=S_{m+1}$.  Equation
\eqref{eq:prod-multinomial} implies that, conditional on $\mathcal{G}$
the columns $M^{c}_{m+1}$ of $M_{m+1}$ are independent $k-$ary
partitions obtained by independent multinomial$-S^{c}$
sampling. Consequently,
\begin{align*}
	P (X^{i}_{m+1}=x^{i}_{m+1} \;\; \forall \; i\in [n]
	\;|\, \mathcal{F}_{m}) 
	&=  P ((M_{m+1}X_{m})^{i}=x^{i}_{m+1} \;\; \forall \; i\in [n]
	\;|\, \mathcal{F}_{m}) \\
	&=  P ((M_{m+1}X_{m})^{i}=x^{i}_{m+1} \;\; \forall \; i\in [n]
	\;|\, \mathcal{G}\vee \sigma (X_{m})) \\
	&=\prod_{i=1}^{n} S_{m+1} (x^{i}_{m+1},X^{i}_{m}) ,
\end{align*}
the second equality by the induction hypothesis and the third by
definition of the probability measure $\mu_{S_{m+1}}$. This proves the
first assertion of the proposition. The equation \eqref{eq:joint law}
follows directly.
\end{proof}

Proposition~\ref{prop:cond indep} shows that for any $n\geq 1$ a
version of the EFCP on $[k]^{[n]}$ can be constructed by first
generating a paintbox sequence $S_{m}$ and then, conditional on
$\mathcal{S}$, running independent, time-inhomogeneous Markov chains
$X^{i}_{m}$ with one-step transition probability matrices
$S_{m}$. From this construction it is evident that a version of the
EFCP on the infinite state space $[k]^{\zz{N}}$ can be constructed by
running countably many conditionally independent Markov chains
$X^{i}_{m}$, and that for any $n\in \zz{N}$ the projection of this
chain to the first $n$ coordinates is a version of the EFCP on
$[k]^{[n]}$.

\section{Random stochastic matrix products}\label{section:rm}

For any EFCP chain $\{X_{m} \}_{m\geq 0}$, Proposition \ref{prop:cond
indep} directly relates the conditional distribution of $X_m$ to the
product $Q_{m}=S_{m}S_{m-1}\dotsb S_{1}$ of i.i.d.\ random stochastic
matrices. Thus, the rates of convergence of these chains are at least
implicitly determined by the contractivity properties of the random
matrix products $Q_{m}.$ The asymptotic behavior of i.i.d. random
matrix products has been thoroughly investigated, beginning with the
seminal paper of Furstenberg and Kesten \cite{FurstenbergKesten1960}:
see \cite{BougerolLacroix} and \cite{goldsheid-margulis} for extensive
reviews. However, the random matrices $S_{i}$ that occur in the
paintbox representation of the $\cp_n(\mu_{\Sigma})$ chain are
not necessarily invertible, so much of the theory developed in
\cite{BougerolLacroix} and \cite{goldsheid-margulis}  doesn't
apply. On the other hand, the random matrices $S_{t}$ are 
\emph{column-stochastic}, and so  the deeper results of
\cite{BougerolLacroix} and \cite{goldsheid-margulis} are not needed
here. In this section we collect the results concerning the
contraction rates of the products $Q_{m}$ needed for the study of the
EFCP chains, and give elementary proofs of these
results.

Throughout this section assume that $\{S_{i} \}_{i\geq 1}$ is a
sequence of independent, identically distributed $k\times k$ random
column-stochastic matrices, with common distribution $\Sigma$, and let 
\[
	Q_{m}=S_{m}S_{m-1}\dotsb S_{1}.
\]

\subsection{Asymptotic Collapse of the Simplex}\label{ssec:collapse}
In the theory of random matrix products, a central role is played by
the induced action on projective space. In the theory of products of
random \emph{stochastic} matrices an analogous role is played by the
action of the matrices on the simplex $\simplexk$.  By definition, the
simplex $\simplexk$ consists of all convex combinations of the unit
vectors $e_1,e_2,\ldots,e_k$ of $\zz{R}^{k}$; since each column of a
$k\times k$ column-stochastic matrix $S\in\ksimplexk$ lies in
$\simplexk$, the mapping $v \mapsto Sv$ preserves $\simplexk$.  This
mapping is \emph{contractive} in the sense that it is Lipschitz
(relative to the usual Euclidean metric on $\zz{R}^{k}$) with
Lipschitz constant $\leq 1$.

The simplex $\simplexk$ is contained in a translate of the
$(k-1)$-dimensional vector subspace $V=V_{k}$ of $\zz{R}^{k}$
consisting of all vectors orthogonal to the vector $\mathbf{1}=
(1,1,\dotsc ,1)^{T}$ (equivalently, the subspace with basis
$e_{i}-e_{i+1}$ where $1\leq i\leq k-1$).  Any stochastic matrix $A$
leaves the subspace $V$ invariant, and hence induces a linear
transformation $A|V:V \rightarrow V$. Since this transformation is
contractive,  its \emph{singular values} are all between
$0$ and $1$. (Recall that the singular values of a $d\times d$ matrix
$S$ are the square roots of the eigenvalues of the nonnegative
definite matrix $S^{T}S$. Equivalently, they are the lengths of the
principal axes of the ellipsoid $S (\mathbb{S}^{d-1})$, where
$\mathbb{S}^{d-1}$ is the unit sphere in $\zz{R}^{d}$.) Denote the
singular values of the restriction $Q_{n}|V$ by
\begin{equation}\label{eq:svs}
	1\geq \lambda_{n,1}\geq \lambda_{n,2}\geq \dotsb \geq \lambda_{n,k-1}\geq 0.
\end{equation}
Because the induced mapping $Q_{n}:\simplexk \rightarrow \simplexk$ is
affine, its Lipschitz constant is just the largest singular value
$\lambda_{n,1}$.

\begin{proposition}\label{proposition:simplex-collapse}
Let $(S_{i})_{i\geq 1}$ be independent, identically distributed $k\times k$
column-stochastic random matrices, and let $Q_{m}=S_{m}S_{m-1}\dotsb
S_{1}$. Then 
\begin{equation}\label{eq:collapse}
	\lim_{m \rightarrow \infty}\text{\rm diameter} ( Q_{m} (\Delta_{k})) =0.
\end{equation} 
if and only if there exists $m\geq 1$ such that with positive
probability the largest singular value $\lambda_{m,1}$ of $Q_{m}|V$ is
strictly less than $1$.  In this case,
\begin{equation}\label{eq:exp-simplex-contraction}
	\limsup_{m \rightarrow \infty} \text{\rm diameter} ( Q_{m}
	(\Delta_{k}))^{1/m} <1 \quad \text{almost surely.}
\end{equation}
\end{proposition}

\begin{proof}
In order that the asymptotic collapse property \eqref{eq:collapse}
holds it is necessary that for some $m$ the largest singular value of
$Q_{m}|V$ be less than one. (If not then for each $m$ there would
exist points $u_{m},v_{m}\in \simplexk$ such that the length of $Q_{m}
(u_{m}-v_{m})$ is at least the length of $u_{m}-v_{m}$; but this would
contradict \eqref{eq:collapse}.) Conversely, if for some $\varepsilon
>0$ the largest singular of $Q_{m}|V$ is less than $1-\varepsilon $
with positive probability then with probability $1$ infinitely many of
the matrix products $S_{mn+m}S_{mn+m-1}\dotsb S_{mn+1}$ have largest
singular value less than $1-\varepsilon$.  Hence, the Lipschitz constant
of the mapping on $\simplexk$ induced by $Q_{mn}$ must converge to $0$
as $n \rightarrow \infty$. In fact even more is true: the asymptotic
fraction as $n \rightarrow \infty$ of blocks where
$S_{mn+m}S_{mn+m-1}\dotsb S_{mn+1}$ has largest singular value
$<1-\varepsilon$ is positive, by strong law of large numbers, and so the Lipschitz constant
of $Q_{mn}:\simplexk \rightarrow \simplexk$ decays exponentially.
\end{proof}

\begin{hypothesis}\label{hypothesis:positive entries} For some integer
$m\geq 1$ the event that all entries of
$Q_{m}$ are positive has positive probability.
\end{hypothesis}

\begin{corollary}\label{corollary:hyp-positivity}
Hypothesis \ref{hypothesis:positive entries} implies the asymptotic collapse
property \eqref{eq:collapse}.
\end{corollary}

\begin{proof}
It is well known that if a stochastic matrix has all entries strictly
positive then its only eigenvalue of modulus $1$ is $1$, and this
eigenvalue is simple (see, for instance, the discussion of the
Perron-Frobenius theorem in the appendix of
\cite{karlin-taylor}). Consequently, if $Q_{m}$ has all entries
positive then $\lambda_{m,1}<1$.
\end{proof}

\subsection{The induced Markov chain on the
simplex}\label{section:induced simplex} The sequence of random matrix
products $(Q_m)_{m\geq 1}$ induce a Markov chain on the simplex
$\simplexk$ in the obvious way: for any initial vector
$Y_0\in\simplexk$ independent of the sequence $(S_m)_{m\geq 0}$, put
\begin{equation}\label{eq:induced simplex}
				   Y_m=Q_mY_0.
\end{equation}
That the sequence $\{Y_{m} \}_{m\geq 0}$ is a Markov chain follows
from the assumption that the matrices $S_{i}$ are i.i.d. Since matrix
multiplication is continuous, the induced Markov chain is Feller
(relative to the usual topology on $\simplexk$). Consequently, since
$\simplexk$ is compact, the induced chain has a stationary
distribution, by the usual Bogoliubov-Krylov argument (see, e.g.,
\cite{petersen}). 

\begin{proposition}\label{proposition:ergodicity-inducedMC}
The stationary distribution of the induced Markov chain on the simplex
is unique if and only if the asymptotic collapse property
\eqref{eq:collapse} holds. 
\end{proposition}

\begin{proof}
[Proof of sufficiency]
Let $\pi$ be a stationary distribution, and let $Y_{0}\sim \pi$ and
$\tilde{Y}_{0}$ be random elements of $\simplexk$ that are independent
of the sequence $\{Q_{m} \}_{m\geq 1}$. Define  $Y_{m}=Q_{m}Y_{0}$ and
$\tilde{Y}_{m}=Q_{m}\tilde{Y}_{0}$. Both sequences $\{Y_{m} \}_{m\geq
0}$ and $\{\tilde{Y}_{m} \}_{m\geq 0}$ are versions of the induced
chain, and since the distribution of $Y_{0}$ is stationary, $Y_{m}\sim
\pi$ for every $m\geq 0$. But the asymptotic collapse property
\eqref{eq:collapse} implies that as $m \rightarrow \infty$,
\[
	d (Y_{m},\tilde{Y}_{m}) \rightarrow 0,
\]
so the distribution of $\tilde{Y}_{m}$ approaches $\pi$ weakly as $m
\rightarrow \infty$. 
\end{proof}

The converse is somewhat more subtle. Recall that the linear subspace
$V=V_{k}$ orthogonal to the vector $\mathbf{1}$ is invariant under
multiplication by any stochastic matrix. Define $U\subset V$ to be the
set of unit vectors $u$ in $V$ such that $\xnorm{Q_{m}u}=\xnorm{u}$
almost surely for every $m\geq 1$. Clearly, the set $U$ is a closed
subset of the unit sphere in $V$, and it is also invariant,
that is, $Q_{m} (U)\subset U$ almost surely.

\begin{lemma}\label{lemma:u}
The set $U$ is empty if and only if the asymptotic collapse property
\eqref{eq:collapse} holds. 
\end{lemma}

\begin{proof}
If \eqref{eq:collapse} holds then $\lim_{m \rightarrow \infty
}\lambda_{m,1}=0$, and so $\xnorm{Q_{m}u} \rightarrow 0$ a.s.\ for every unit vector
$u\in V$. Thus, $U=\emptyset$. 

To prove the converse statement, assume that the asymptotic collapse
property \eqref{eq:collapse} \emph{fails.} Then by
Proposition~\ref{proposition:simplex-collapse}, for each $m\geq 1$ the
largest singular value of $Q_{m}|V$ is $\lambda_{m,1}=1$, and
consequently there exist (possibly random) unit vectors $v_{m}\in V $
such that $\xnorm{Q_{m}v_{m}}=1$. Since each matrix $S_{i}$ is
contractive, it follows that $\xnorm{Q_{m}v_{m+n}}=1$ for all $m,n\geq
1$. Hence, by the compactness of the unit sphere and the continuity of
the maps $Q_{m}|V$, there exists a possibly random unit vector $u$
such that $\xnorm{Q_{m}u}=1$ for every $m\geq 1$.

We will now show that there exists a \emph{non-random} unit vector $u$
such that $\xnorm{Q_{m}u}=1$ for every $m$, almost surely. Suppose to
the contrary that there were no such $u$.  For each unit vector $u$,
let $p_{m} (u)$ be the probability that $\xnorm{Q_{m}u}<1$. Since the
matrices $S_{m}$ are weakly contractive, for any unit vector $u$ the
events $\xnorm{Q_{m}u}=1$ are decreasing in $m$, and so $p_{m} (u)$ is
non-decreasing. Hence, by a subsequence argument, if for every $m\geq
1$ there were a unit vector $u_{m}$ such that $p_{m}(u_{m})=0$, then
there would be a unit vector $u$ such that $p_{m} (u)=0$ for every
$m$. But by assumption there is no such $u$; consequently, there must
be some finite $m\geq 1$ such that $p_{m} (u)>0$ for every unit
vector.

 For each fixed $m$, the function $p_{m} (u)$ is lower semi-continuous
(by the continuity of matrix multiplication), and therefore attains a
minimum on the unit sphere of $V$. Since $p_{m}$ is strictly positive,
it follows that there exists $\delta >0$ such that $p_{m} (u)\geq
\delta$ for every unit vector $u$.  But if this is the case then there
can be no random unit vector $u$ such that $\xnorm{Q_{m}u}=1$ for
every $m\geq 1$, because for each $m$ the event that
$\xnorm{Q_{m+1}u}<\xnorm{Q_{m} u}$ would have conditional probability
(given $S_{1},S_{2},\dotsc ,S_{m}$) at least $\delta$.
\end{proof}

\begin{proof}
[Proof of necessity in
Proposition~\ref{proposition:ergodicity-inducedMC}]
If the asymptotic collapse property
\eqref{eq:collapse} {fails}, then by Lemma~\ref{lemma:u} there exists
a unit vector $u\in V$ such that $\xnorm{Q_{m}u}=1$ for all $m\geq 1$,
almost surely.  Hence, since $\simplexk$ is contained in a translate
of $V$, there exist distinct $\mu ,\nu\in \simplexk$ such that 
$\xnorm{Q_{m} (\mu -\nu)}=\xnorm{\mu -\nu}$ for all $m\geq 1$, a.s. 
By compactness, there exists such a pair $(\mu ,\nu)\in \Delta^{2}_k$
for which $\xnorm{\mu -\nu}$ is maximal. 
Fix such a pair $(\mu ,\nu)$, and let $A\subset \Delta^2_k$ be the
set of all pairs $(y,z)$ such that  
\[
	\xnorm{S_{1}y -S_{1}z}=\xnorm{\mu -\nu} \quad \text{a.s.}
\]
Note that the set $A$ is closed, and consequently
compact. Furthermore, because $\mu ,\nu$ have been chosen so that
$\xnorm{\mu -\nu}$ is maximal, for any pair $(y,z)\in A$ the points $y$ and
$z$ must both lie in the  boundary $\partial \simplexk$ of the simplex.

Define $Y_{m}=Q_{m}\mu $, $Z_{m}=Q_{m}\nu$, and $R_{m}=
(Y_{m}+Z_{m})/2$.  By construction, for each $m\geq 0$ the pair $(
Y_{m},Z_{m})$ lies in the set $A$.  The sequence $(Y_{m},Z_{m},R_{m})$
is a $\Delta^{3}_k-$valued Markov chain, each of whose projections on
$\simplexk$ is a version of the induced chain. Since $\Delta^{3}_k$
is compact, the Bogoliubov-Krylov argument implies that the Markov
chain $(Y_{m},Z_{m},R_{m})$ has a stationary distribution $\lambda$ whose
projection $\lambda_{Y,Z}$ on the first two coordinates is supported
by $A$. Each of the marginal distributions $\lambda_{Y}$,
$\lambda_{Z}$, and $\lambda_{R}$ is obviously stationary for the
induced chain on the simplex, and both $\lambda_{Y}$ and $\lambda_{Z}$
have supports contained in $\partial \simplexk $. Clearly, if
$(Y,Z,R)\sim \lambda $ then  $R= (Y+Z)/2$.

We may assume that $\lambda_{Y}=\lambda_{Z}$, for otherwise there is
nothing to prove.  We claim that $\lambda_{R}\not =\lambda_{Y}$. To
see this, let $D$ be the minimal integer such that $\lambda_{Y}$ is
supported by the union $\partial_{D}\simplexk$ of the $D-$dimensional
faces of $\simplexk$.  If $(Y,Z,R)\sim \lambda$, then $Y\not
=Z$, since $\lambda_{Y,Z}$ has support in $A$. Consequently, $(Y+Z)/2$
is contained in the interior of a $(D+1)-$dimensional face of
$\simplexk$. It follows that $\lambda_{R}\not =\lambda_{Y}$.

\end{proof}

\begin{remark}\label{remark:recurrence} \emph{Recurrence Times.}
Assume that the asymptotic collapse property \eqref{eq:collapse}
holds, and let $\nu$ be the unique stationary distribution for the
induced chain on the  simplex. Say that a point $v$ of the simplex is
a \emph{support point} of $\nu$ if $\nu$ gives positive probability to
every open neighborhood of $v$. Fix such a neighborhood $U$, and let
$\tau$ be the first time $m\geq 1$ that $Y_{m}\in U$. Then there
exists $0<r=r_{U}<1$ such that for all $m\geq  1$,
\[
	P\{\tau >m \}\leq r^{m},
\]
regardless of the initial state $Y_{0}$ of the induced chain. To see
this, observe that because $\nu (U)>0$ there exists $m$ such that the
event  $Q_{m} (\simplexk)\subset U$ has positive
probability. Consequently, because the matrices $S_{i}$ are i.i.d.,
the probability that $Q_{mn} (\simplexk)\not \subset U$ for all
$n=1,2,\dotsc ,N$ is exponentially decaying in $N$.
\end{remark}

\begin{remark}\label{remark:asymptoticFreqs}
\emph{Relation between the induced chain on $\simplexk$ and the EFCP.}
Let $\{X_{m} \}_{m\geq 0}$ be a version of the EFCP on $[k]^{\zz{N}}$
with paintbox sequence $\{S_{m} \}_{m\geq 1}$. By
Proposition~\ref{prop:cond indep}, the individual coordinate sequences
$\{X^{i}_{m} \}_{m\geq 0}$ are conditionally independent given
$\mathcal{G}=\sigma (X_{0},S_{1},S_{2},\dotsc)$, and for each $i$ the
sequence $\{X^{i}_{m} \}_{m\geq 0}$ evolves as a time-inhomogeneous
Markov chain with one-step transition probability matrices $S_{m}$.
Consequently, by the strong law of large numbers, if the initial state
$X_{0}$ has the property that the limiting frequencies of all colors
$r\in [k]$ exist with probability one (as would be the
case if the initial distribution is exchangeable), then this property
persists for all times $m\geq 1$. In this case, the sequence $\{Y_{m}
\}_{m\geq 0}$, where  $Y_{m}$ is the vector  of
limiting color frequencies in the $m$th generation, is a version of
the induced Markov chain on the simplex $\simplexk$.  Moreover, the
$j$th column of the 
stochastic matrix $S_{m}$ coincides with the limit frequencies of
colors in $X_{m}$ among those indices $i\in \zz{N}$ such that
$X^{i}_{m-1}=j$. Thus, the paintbox sequence can be recovered (as a
measurable function) from the EFCP.
\end{remark}

\subsection{Asymptotic Decay Rates}\label{ssec:asymptotic-decay}

 Lebesgue measure on $\simplexk$ is obtained by
translating Lebesgue measure on $V$ (the choice of Lebesgue measure
depends on the choice of basis for $V$, but for any two choices the
corresponding Lebesgue measures differ only by a scalar multiple).
The $k-$fold product of Lebesgue measure on $\simplexk$ will be
referred to as \emph{Lebesgue measure} on $\ksimplexk$.


\begin{hypothesis}\label{hypothesis:L2Density}
The distribution $\Sigma$ of the random stochastic matrix $S_1$ is absolutely
continuous with respect to Lebesgue measure on $\stochk$ and has a
density of class $L^{p}$ for some $p>1$.
\end{hypothesis}

Hypothesis \ref{hypothesis:L2Density} implies that the conditional
distribution of the $i$th column of $S_{1}$, given the other $k-1$
columns, is absolutely continuous relative to Lebesgue measure on
$\simplexk$. Consequently, the conditional probability that it is a
linear combination of the other $k-1$ columns is $0$. Therefore, the matrices
$S_{t}$ are almost surely nonsingular, and so the Furstenberg theory
(\cite{BougerolLacroix}, chapters 3--4) applies.  Furthermore, under
Hypothesis~\ref{hypothesis:L2Density} the entries of $S_{1}$ are
positive, with probability $1$. Thus,
Hypothesis~\ref{hypothesis:L2Density} implies
Hypothesis~\ref{hypothesis:positive entries}.

\begin{proposition}\label{proposition:m2density}
Under Hypothesis~\ref{hypothesis:L2Density}, 
\begin{equation}\label{eq:finite-log-expectation}
	E|\log |\det S_{1}||<\infty,
\end{equation}
and consequently
\begin{equation}\label{eq:exp-decay-determinant}
	\lim_{n \rightarrow \infty } (\det (Q_{n}|V))^{1/n}=e^{\kappa}
	\quad \text{where} \quad 
	\kappa = E\log \det S_{1}.
\end{equation}
\end{proposition}

\begin{note}\label{note:exponents}
The determinant of $S_1$ is the volume of the polyhedron
$S_1[0,1]^{k}$,  which is $\sqrt{k}$ times the volume of the
$(k-1)$-dimensional polyhedron with vertices $S_1e_{i}$, where
$1\leq i\leq k$. The volume of this
$(k-1)-$dimensional polyhedron is the determinant of the restriction
$S_1|V$. Consequently,
\[
	\det S_1|V=\prod_{i=1}^{k-1}\lambda_{1,i}.
\]
\end{note}

\begin{proof}
The assertion \eqref{eq:exp-decay-determinant} follows from
\eqref{eq:finite-log-expectation}, by the strong law of large numbers,
since the determinant is multiplicative. It remains to prove
\eqref{eq:finite-log-expectation}.  Fix $\varepsilon >0$, and consider
the event $\det S_1<\varepsilon$. This event can occur only if the
smallest singular value of $S_1$ is less than $\varepsilon^{1/k}$, and
this can happen only if one of the vectors $S_1e_{i}$ lies within
distance $\varepsilon^{1/k}$ (or so) of a convex linear combination of
the remaining $S_1e_{j}$.

 The vectors $S_1e_{i}$, where $i\in [k]$, are the columns of $S_1$, whose
distribution is assumed to have a $L^p$ density $f (M)$
with respect to Lebesgue measure $dM$ on $\stochk$. Fix an integer
$m\geq 1$, and consider the subset $B_{m}$ of $\stochk$ consisting
of all $k\times k$ stochastic matrices $M$ such that the $i$th column
$Me_{i}$ lies within distance $e^{-m}$ of the set of all convex
combinations of the remaining columns $Me_{j}$. Elementary geometry
shows that the set $B_{m}$ has Lebesgue measure $\leq Ce^{-m}$, for
some constant $C=C_{k}$ depending on the dimension but not on $m$ or
$i$. Consequently, by the H\"{o}lder inequality, for a suitable
constant $C'=C'_{k}<\infty$,
\begin{align*}
		E|\log |\det S_1|| &\leq C'\sum_{m=0}^{\infty} (m+1)
		\int_{B_{m}} f (M) \,dM \\
		&\leq  C'\sum_{m=0}^{\infty} (m+1) 
		\left\{\int_{B_{m}} 1\,dM \right\}^{1/q}
		\left\{\int f (M)^{p}\,dM \right\}^{1/p}\\
		&\leq  C'\sum_{m=0}^{\infty} (m+1) e^{-m/q}
		\left\{\int f (M)^{p}\,dM \right\}^{1/p} <\infty 
\end{align*}
where $1/p+1/q=1$.  In fact, this also shows that $\log|\det S_1|$ has
finite moments of all orders, and even a finite moment generating
function in a neighborhood of $0$.
\end{proof}

\begin{proposition}\label{proposition:logNorm}
Under hypotheses \ref{hypothesis:L2Density},
\begin{equation}\label{eq:2ndLyapunovExponent}
	\lim_{n \rightarrow \infty} \lambda_{n,1}^{1/n} :=\lambda_{1}\quad\mbox{exists a.s.}
\end{equation}
Moreover, the limit $\lambda_{1}$  is constant and
satisfies $0<\lambda_{1}<1$. 
\end{proposition}

\begin{remark}\label{rmk:sv}
It can be shown that the Lyapunov exponents of the sequence $Q_{m}$
are the same as those of $Q_{m}|V$, but with one additional Lyapunov
exponent $0$. Thus, $\log \lambda_{1}$ is the \emph{second} Lyapunov
exponent of the sequence $Q_{m}$. 
\end{remark}

\begin{remark}\label{remark:furstenberg}
Hypothesis~\ref{hypothesis:L2Density} implies that the distribution of
$S_{1}$ is \emph{strongly irreducible} (cf. \cite{BougerolLacroix},
ch. 3), and so a theorem of Furstenberg implies that the top two
Lyapunov exponents of the sequence $Q_{m}$ are distinct. However,
additional hypotheses are needed to guarantee that
$\lambda_{1}>0$. This is the main point of
Propositions~\ref{proposition:m2density}--\ref{proposition:logNorm}.
\end{remark}

\begin{proof}
[Proof of Proposition \ref{proposition:logNorm}] The almost sure
convergence follows from the Furstenberg-Kesten theorem
\cite{FurstenbergKesten1960} (or alternatively, Kingman's subadditive
ergodic theorem \cite{kingman}), because the largest singular value of
$Q_{n}|V$ is the matrix norm of $Q_{n}|V$, and the matrix norm is
sub-multiplicative.  That the limit $\lambda_{1}$ is constant follows
from the Kolmogorov $0-1$ law, because if the matrices $S_{j}$ are
nonsingular (as they are under the hypotheses on the distribution of
$S_1$) the value of $\lambda_{1}$ will not depend on any initial
segment $S_{m}S_{m-1}\dotsb S_1$ of the matrix products.

 That $\lambda_{1}<1$ follows from assertion
\eqref{eq:exp-simplex-contraction} of
Proposition~\ref{proposition:simplex-collapse}, because
Hypothesis~\ref{hypothesis:positive entries} implies that there is a
positive probability $\eta>0$ that all entries of $S_1$ are at least
$\varepsilon>0$, in which case $S_1$ is strictly contractive on
$\simplexk$ -- and hence also on $V$ -- with contraction factor
$\theta =\theta(\varepsilon)<1$ (\cite{Golubitsky1975}, proposition
1.3).

Finally, the assertion that $\lambda_{1}>0$  follows from
Proposition \ref{proposition:m2density},  because for any stochastic
matrix each  singular value is bounded below by the determinant.
\end{proof}

\begin{corollary}\label{corollary:contraction}
Under hypotheses \ref{hypothesis:L2Density},
\[
 \lim_{n \rightarrow
	\infty} \max_{i\not =j} 
	\xnorm{Q_{n}e_{i}-Q_{n}e_{j}}^{1/n} =
	\lambda_{1} \quad \text{almost surely.}
\]
\end{corollary}

\begin{proof}
The lim sup of the maximum cannot be greater than
$\lambda_{1}$, because for each $n$ 
the singular value $\lambda_{n,1}$ of $Q_{n}|V$ is just the matrix
norm $\xnorm{Q_{n}}$. To prove the reverse inequality, assume the
contrary. Then there is a subsequence $n=n_{m} \rightarrow \infty$
along which
\[
		\limsup_{m \rightarrow \infty} \max_{i\not =j}
	\xnorm{Q_{n}e_{i}-Q_{n}e_{j}}^{1/n} <\lambda_{1}-\varepsilon 
\]
for some $\varepsilon >0$. Denote by $u=u_{n}\in V$ the unit vector
that maximizes  $\xnorm{Q_{n}u}$. Because the vectors $e_{i}-e_{i+1}$
form a basis of $V$, for each $n$ the vector $u_{n}$ is a linear
combination $u_{n}=\sum_{i}a_{ni} (e_{i}-e_{i+1})$, and because each
$u_{n}$ is a unit vector, the coefficients $a_{ni}$ are uniformly
bounded by (say) $C$ in magnitude. Consequently,
\[
	\xnorm{Q_{n}u_{n}}\leq C \sum_{i}\xnorm{Q_{n}(e_{i}-e_{i+1})}.
\]
This implies that along the subsequence $n=n_{m}$ we have
\[
		\limsup_{m \rightarrow \infty}
		\xnorm{Q_{n}u_{n}}^{1/n} <\lambda_{1}-\varepsilon .
\]
But this contradicts the fact that $\xnorm{Q_{n}|V}^{1/n}\rightarrow
\lambda_{1}$ from proposition \ref{proposition:logNorm}.

\end{proof}

\begin{remark}\label{remark:min}
It can be also be shown  that 
\[
		\lim_{n \rightarrow \infty} \min_{i\not =j}
	\xnorm{Q_{n}e_{i}-Q_{n}e_{j}}^{1/n} =\lambda_{1}.
\]
This, however, will not be needed for the results of
section~\ref{section:convergence}.
\end{remark}

\begin{remark}\label{note:contraction}
The argument used to prove that $\lambda_{1}<1$ in the proof of
Proposition~\ref{proposition:logNorm} also proves that even if
Hypothesis~\ref{hypothesis:L2Density} fails, if the
distribution of $S_1$ puts positive weight on the set of stochastic
matrices with all entries at least $\varepsilon$, for some
$\varepsilon >0$, then
\begin{equation}\label{eq:expContraction}
	\limsup_{n \rightarrow \infty}\max_{i\not
	=j}\xnorm{Q_{n}e_{i}-Q_{n}e_{j}}^{1/n} <1.
\end{equation}
Hypothesis~\ref{hypothesis:L2Density} guarantees that the sequence
$\xnorm{Q_{n}e_{i}-Q_{n}e_{j}}^{1/n}$ has a limit, and that the limit
is positive. When Hypothesis~\ref{hypothesis:L2Density} fails, the
convergence in \eqref{eq:expContraction} can be super-exponential
(i.e., the limsup in \eqref{eq:expContraction} can be $0$).  For
instance, this is the case if for some rank-1 stochastic matrix $A$
with all entries positive there is positive probability that $S_1=A$.
\end{remark}

\section{Convergence to stationarity of EFCP
chains}\label{section:convergence} 

Assume throughout this section that $\{X_{m} \}_{m\geq 1}$ is an EFCP
on $[k]^{[n]}$ or $[k]^{\zz{N}}$ with directing measure
$\mu_{\Sigma}$, as defined by \eqref{eq:matrix mixture}. Let
$S_{1},S_{2},\dotsc$ be the associated paintbox sequence: these are
i.i.d. random column-stochastic matrices with distribution $\Sigma$.
Proposition~\ref{prop:cond indep} shows that the joint distribution of
the coordinate variables $X^{i}_{m}$ of an EFCP chain with paintbox
sequence $\{S_{i} \}_{i\geq 1}$ is controlled by the random matrix
products $Q_{m}=S_{m}S_{m-1}\dotsb S_{1}$. In this section we use this
fact together with the results concerning random matrix products
recounted in section \ref{section:rm} to determine the mixing rates of
the restrictions $\{X^{[n]}_{m} \}_{m\geq 1}$ of EFCP chains to the
finite configuration spaces $[k]^{[n]}$.

\subsection{Ergodicity}\label{ssec:ergodicity}
An EFCP chain  need not be ergodic:  for instance, if each $S_{i}$ is
the identity matrix then every state is absorbing and
$X^{i}_{m}=X^{i}_{0}$ for every $m\geq 1$ and every $i\in \zz{N}$.
More generally, if the random matrices $S_{i}$ are all permutation
matrices then the \emph{unlabeled} partitions of $\zz{N}$ induced by
the labeled partitions $X_{m}$ do not change with $m$, and so the
restrictions $X^{[n]}_{m}$ cannot be ergodic. The failure of
ergodicity in these examples stems from the fact that the matrix
products $Q_{m}$ do not contract the simplex $\Delta_{k}$. 

\begin{proposition}\label{proposition:induced-stationary}
Let $\lambda$ be any stationary distribution for the induced Markov
chain on the simplex. Then for each $n\in \zz{N}\cup \{\infty \}$ the
$\lambda -$mixture $\varrho_{\lambda }^{n}$ of the product multinomial
measures on $[k]^{[n]}$ is stationary for the EFCP chain on
$[k]^{[n]}$.
\end{proposition}

\begin{note}\label{note:lambda-mixtures}
Recall that the product-multinomial measures $\varrho_{s}^{n}$ are
defined by equation \eqref{eq:multi-s}; the $\lambda -$mixture is
defined to be the average
\[
	\varrho_{\lambda}^{n}=\int_{\simplexk}
	\varrho_{s}^{n}\,\lambda (ds). 
\]
Thus, a random configuration $X\in [k]^{[n]}$ with distribution
$\varrho_{\lambda }^{n}$ can be obtained by first choosing $s\sim
\lambda$, then, conditional on $s$, independently assigning colors to
the coordinates $i\in [n]$ by sampling from the
$\varrho_{s}$ distribution.
\end{note}

\begin{proof}
This is an immediate consequence of Proposition \ref{prop:cond indep}.
\end{proof}

\begin{proposition}\label{proposition:ergodicityces-condition}
Assume that with probability one the random matrix products $Q_{m}$
asymptotically collapse the simplex $\Delta_{k}$, that is,
\begin{equation}\label{eq:singletonLims2}
	\lim_{m \rightarrow \infty}\text{\rm diameter} ( Q_{m} (\Delta_{k})) =0.
\end{equation} 
Then for each $n\in \zz{N}$ the corresponding ECFP chain
$\{X^{[n]}_{m} \}_{m\geq 0}$ on $[k]^{[n]}$ is ergodic, i.e., has a
unique stationary distribution. Conversely, if for some $n\geq 1$ the
EFCP chain $\{X^{[n]}_{m} \}_{m\geq 0}$ is ergodic then the asymptotic
collapse property \eqref{eq:singletonLims2} must hold.
\end{proposition}

\begin{proof}
Fix $n\geq 1$. By Propositions \ref{proposition:ergodicity-inducedMC} and 
\ref{proposition:induced-stationary}, there exists
at least one stationary distribution $\pi$. Let $\{X_{m} \}_{m\geq 0}$
and $\{\tilde{X}_{m} \}_{m\geq 0}$ be conditionally independent
versions of the EFCP given the (same) paintbox sequence
$(S_{i})_{i\geq 1}$, with $\tilde{X}_{0}\sim \pi$ and $X_{0}\sim \nu$
arbitrary. Then for any time $m\geq 1$ the conditional distributions
of $X_{m}$ and $ \tilde{X}_{m}$ given the paintbox sequence can be
recovered from the formula \eqref{eq:joint law} by integrating out
over the distributions of $X_{0}$ and $\tilde{X}_{0}$,
respectively. But under the hypothesis \eqref{eq:singletonLims2}, for
large $m$ the columns of $Q_{m}$ are, with high probability, nearly
identical, and so for large $m$ the products
\[
	\prod_{i=1}^{n}	Q_m(x_m^i,X_0^i)
	\quad \text{and} \quad \prod_{i=1}^{n}
	Q_m(x_m^i,\tilde{X}_0^i)
\]
will be very nearly the same. It follows, by integrating over all
paintbox sequences, that the unconditional distributions of $X_{m}$
and $\tilde{X}_{m}$ will be nearly the same when $m$ is large. This
proves that the stationary distribution $\pi $ is unique and that as
$m \rightarrow \infty$ the distribution of $X_{m}$ converges to $\pi$.

By proposition \ref{proposition:ergodicity-inducedMC}, if the
asymptotic collapse property \eqref{eq:singletonLims2} {fails} then
the induced Markov chain on the simplex has at least two distinct
stationary distributions $\mu ,\nu$. By
Proposition~\ref{proposition:induced-stationary}, these correspond to
different stationary distributions for the EFCP.

\end{proof}

\subsection{Mixing rate and cutoff for EFCP
chains}\label{section:cutoff} We measure distance to stationarity
using the total variation metric \eqref{eq:definitionTV}. Write
$\mathcal{D}(X_m)$ to denote the distribution of $X_m$.  In general,
the distance $\xnorm{\mathcal{D}(X_m)-\pi}_{TV}$ will depend on the
distribution of the initial state $X_{0}$. The $\varepsilon
-$\emph{mixing time} is defined to be the number of steps needed to
bring the total variation distance between $\mathcal{D} (X_m)$ and
$\pi$ below $\varepsilon$ for \emph{all} initial states $x_{0}$:
\begin{equation}\label{eq:tmix}
	t_{\text{\textsc{mix}}} (\varepsilon ) =t_{\text{\textsc{mix}}}^{(n)} (\varepsilon ) =
	\min \{m\geq 1 \,:\, \max_{x_{0}} \xnorm{\mathcal{D}
	(X_m)-\pi}_{TV}<\varepsilon \}.
\end{equation}

\begin{theorem}\label{theorem:ub}
Assume that with probability one the random matrix products
$Q_{m}=S_{m}S_{m-1}\dotsb S_{1}$ asymptotically collapse the simplex
$\Delta_{k}$, that is, relation \eqref{eq:collapse} holds. Then for a
suitable constant $K=K_{\Sigma}<\infty$ depending only on the
distribution $\Sigma$ of $S_{1}$, the mixing times of the
corresponding EFCP chains on the finite state spaces $[k]^{[n]}$
satisfy
\begin{equation}
 t_{\text{\textsc{mix}}}^{(n)}(\varepsilon)\leq K\log n.
\label{eq:logUpper}
\end{equation}
\end{theorem}

\begin{remark}\label{remark:fast-convergence}
In some cases the mixing times will be of smaller order of magnitude
than $\log n$. Suppose, for instance, that for some $m\geq 1$ the
event that the matrix $Q_{m}$ is of rank $1$ has positive probability.  (This
would be the case, for instance, if the columns of $S_{1}$ were
independently chosen from a probability distribution on $\simplexk$
with an atom.) Let $T$ be the least $m$ for which this is the case;
then $T<\infty$ almost surely, since matrix rank is
sub-multiplicative, and $Q_{m} (\simplexk)$ is a singleton for any
$m\geq T$.  Consequently, for any elements $a,b,c\in [k]$,
\[
	Q_{m} (a,b)=Q_{m} (a,c) \quad \text{if} \;\; T\leq m.
\]
Hence, if $\{X_{m} \}_{m\geq 0}$ and $\{\tilde{X}_{m} \}_{m\geq
0}$ are versions of the EFCP with different initial conditions $X_{0}$
and $\tilde{X}_{0}$, but with the same paintbox sequence $S_{m}$, then
by Proposition \ref{prop:cond indep}, $X_{m}$ and $\tilde{X}_{m}$ have
the same conditional distribution, given $\sigma (S_{i})_{i\geq 1}$,
on the event $T\leq m$. It follows that the total variation
distance between the \emph{unconditional} distributions of $X_{m}$ and
$\tilde{X}_{m}$ is no greater than $P\{T>m \}$. Thus, for any $n\in
\zz{N}$, the EFCP mixes in $O (1)$ steps, that is, for any
$\varepsilon >0$ there exists $K_{\varepsilon}<\infty$ such that for
all $n$,
\[
	t_{\text{\textsc{mix}}}^{(n)}(\varepsilon)\leq K_{\varepsilon}. 
\]
\end{remark}

\begin{proof}
[Proof of Theorem \ref{theorem:ub}] (A) Consider first the special case
where for some $\delta >0$ every entry of $S_{1}$ is at least
$\delta$, with probability one.  It then follows that no entry of
$Q_{m}$ is smaller than $\delta$.  By Proposition
\ref{proposition:simplex-collapse}, if \eqref{eq:collapse} holds then
the diameters of the sets $Q_{m} (\simplexk)$ shrink exponentially
fast: in particular, for some (nonrandom) $\varrho <1$,
\begin{equation}\label{eq:shrink-rate}
	\text{diameter} (Q_{m} (\simplexk))<\varrho^{m} 
\end{equation}
eventually, with probability $1$. 

Let $\{X_{m} \}_{m\geq 0}$ and $\{\tilde{X}_{m} \}_{m\geq 0}$ be
versions of the EFCP on $[k]^{[n]}$ with different initial conditions
$X_{0}$ and $\tilde{X}_{0}$, but with the same paintbox sequence
$S_{m}$. By Proposition~\ref{prop:cond indep}, the conditional
distributions of $X_{m}$ and $\tilde{X_{m}}$ given the paintbox
sequence are product-multinomials:
\begin{align}\label{eq:cds}
	 P ({X}^{i}_{m} =x^{i} \;\;\text{for each}\; i\in
	 [n]\,|\,\mathcal{S})&= \prod_{i=1}^{n} Q_{m}
	(x^{i}_{m},{X}^{i}_{0}) \quad \text{and}\\
\notag 	P (\tilde{X}^{i}_{m} =x^{i} \;\;\text{for each}\; i\in
	 [n]\,|\,\mathcal{S})&= \prod_{i=1}^{n} Q_{m}
	(x^{i}_{m},\tilde{X}^{i}_{0}).	 
\end{align}
Since the multinomial distributions $Q_{m} (\cdot, \cdot)$ assign
probability at least $\delta >0$ to every color $j\in [k]$,
Corollary~\ref{cor:prod-multiTV} implies that for any $\varepsilon
>0$, if $m=K\log n$, where $K> -1/ ( 2\log \varrho)$, then for all
sufficiently large $n$ the total variation distance between the
conditional distributions of $X_{m}$ and $\tilde{X}_{m}$ will differ
by $\varepsilon $ on the event \eqref{eq:shrink-rate} holds. Since
\eqref{eq:shrink-rate} holds eventually, with probability one, the
inequality \eqref{eq:logUpper} now follows by Lemma~\ref{lemma:mix}.
 
\bigskip (B) The general case requires a bit more care, because if the
entries of the matrices $Q_{m}$ are not bounded below then the
product-multinomial distributions \eqref{eq:cds} will not be bounded
away from $\partial \simplexk$, as required by Corollary~\ref{cor:prod-multiTV}.

Assume first that for some $m\geq 1$ there is positive probability
that $Q_{m} (\simplexk)$ is contained in the {interior} of
$\simplexk$.  Then for some $\delta >0$ there is probability at least
$\delta$ that every entry of $Q_{m}$ is at least $\delta$.
Consequently, for any $\alpha >0$ and any $K>0$, with probability
converging to one as $n \rightarrow \infty$, there will exist $m\in
[K\log n, K (1+\alpha) \log n]$ (possibly random) such that every
entry of $Q_{m}$ is at least $\delta$. By \eqref{eq:shrink-rate} the
probability that the diameter of $Q_{m} (\simplexk)$ is less than
$\varrho^{m}$ converges to $1$ as $m \rightarrow \infty$. It then
follows from Corollary~\ref{cor:prod-multiTV}, by the same argument as
in (A), that if $K>-1/ (2\log \varrho )$ then the total variation
distance between the conditional distributions of $X_{m}$ and
$\tilde{X_{m}}$ will differ by a vanishingly small amount. Since total
variation distance decreases with time, it follows that the total variation
distance between the
conditional distributions of $X_{K+K\alpha}$ and
$\tilde{X}_{K+K\alpha}$ are also vanishingly small. Consequently, the
distance between the unconditional distributions is also small, and so
\eqref{eq:logUpper} follows, by Lemma~\ref{lemma:mix}.

\bigskip 
(C) Finally, consider the case where $Q_{m} (\simplexk)$ intersects
$\partial \simplexk$ for every $m$, with probability one. 
Recall (Proposition~\ref{proposition:induced-stationary}) that if the
asymptotic collapse property \eqref{eq:collapse} holds then the
induced Markov chain $Y_{m}$ on the simplex has a unique stationary
distribution $\nu$. If there is no $m\in \zz{N}$ such that $Q_{m} (\simplexk)$
is contained in the  interior of $\simplexk$, then the support of
$\nu$ must be contained in the boundary $\partial \simplexk$. Fix a support
point $v$, and let $m$ be sufficiently large that
\eqref{eq:shrink-rate} holds. Since $Q_{m} (\simplexk)$ must intersect
$\partial \simplexk$, it follows that for any coordinate $a\in [k]$
such that $v_{a}=0$ (note that there must be at least one such $a$,
because $v\in \partial \simplexk$), the $a$th coordinate
$(Q_{m}y)_{a}$ of any point in the image $Q_{m} (\simplexk)$ must be
smaller than $\varrho^{m}$. If $K$ is chosen sufficiently large and
$m\geq K\log n$, then $\varrho^{m}<n^{-2}$; hence, by
Proposition~\ref{prop:cond indep}, 
\[
	P (X^{i}_{m}=a \;\;\text{for some}\; i\in [n]\,|\, \sigma
	(S_{l})_{l\geq 1}) \leq n\cdot n^{-2}= n^{-1} \rightarrow
	0,
\]
and similarly for $\tilde{X}_{m}$. Therefore, the contribution to the
total variation distance between the conditional distributions of
$X_{m}$ and $\tilde{X}_{m}$ from states $x^{1}x^{2}\dotsb x^{n}$ in
which the color $a$ appears at least once is vanishingly small. But
for those states for which no such color appears, the factors $Q_{m}
(a,b)$ in \eqref{eq:cds}  will be bounded below by the minimum nonzero
entry of $v$, and the result will  follow by a 
routine modification of the argument in (B) above.
\end{proof}

Parts (A)-(B) of the foregoing proof provide an explicit bound in the
special case where $Q_{m} (\simplexk)$ is contained in the  interior
of $\simplexk$ with positive probability.

\begin{corollary}\label{corollary:explicitUB}
Assume that with probability one the random matrix products
$Q_{m}=S_{m}S_{m-1}\dotsb S_{1}$ asymptotically collapse the simplex
$\Delta_{k}$, so that for some $0<\varrho <1$,
\[
		\text{\rm diameter} (Q_{m} (\simplexk))<\varrho^{m} 
\]
for all sufficiently large $m$, with probability $1$. Assume also
that with positive probability $Q_{m} (\simplexk)$ is contained in the
interior of $\simplexk$, for some $m\geq 1$.  Then for any $K>-1/
(2\log \varrho )$ the
bound \eqref{eq:logUpper} holds for all sufficiently large $n$.
\end{corollary}

\begin{theorem}\label{theorem:cutoff}
Assume that the paintbox distribution $\Sigma$ satisfies hypothesis
\ref{hypothesis:L2Density}. Then the corresponding EFCP chains exhibit the {\em cutoff
phenomenon}, that is, for all $\varepsilon,\delta\in (0,1/2)$, if $n$ is
sufficiently large, then
\begin{equation}
\label{eq:cutoff}
	 (\theta-\delta)\log n\leq
	 t_{\text{\textsc{mix}}}^{(n)}(1-\varepsilon)\leq
	 t_{\text{\textsc{mix}}}^{(n)}(\varepsilon)\leq(\theta+\delta)\log
	 n,
\end{equation} 
where 
\begin{equation}\label{eq:cutoff-constant}\theta=-1/(2\log\lambda_1)
\end{equation}
and $\lambda_1$ is the second Lyapunov exponent of the sequence
$Q_{m}$, that is,  as in proposition \eqref{proposition:logNorm}.
\end{theorem}

\begin{proof}
[Proof of the Upper Bound $t_{\text{\textsc{mix}}} (\varepsilon )\leq
(\theta+\delta )\log n$]  Because the distribution of
$S_{1}$ is absolutely continuous with respect to Lebesgue measure,
there is positive probability that all entries of $S_{1}=Q_{1}$ are
positive, and so there is positive probability that $Q_{1}
(\simplexk)$ is contained in the interior of $\simplexk$.  Therefore,
Corollary~\ref{corollary:explicitUB} applies. But Proposition~\ref{proposition:logNorm}
and Corollary~\ref{corollary:contraction} implies that, under Hypothesis~\ref{hypothesis:L2Density}, that $\varrho
=\lambda_{1}$. 

\end{proof}

\begin{proof}[Proof of the Lower Bound $t_{\text{\textsc{mix}}}(\varepsilon)\geq
(\theta -\delta  )\log n$] It suffices to show that there exist initial
states $x_{0},\tilde{x}_{0}$ such that if $\{X_{t} \}_{t\geq 0}$ and
$\{\tilde{X}_{t} \}_{t\geq 0} $ are versions of the EFCP with initial
states $X_{0}=x_{0}$ and $\tilde{X}_{0}=\tilde{x}_{0}$, respectively,
then the distributions of $X_{m}$ and $\tilde{X}_{m}$ have  total
variation distance near $1$ when $m\leq(\theta-\delta)\log n$. 
The proof will rely on Corollary~\ref{corollary:contraction},
according to which there is a possibly  random pair of indices $i\not
=j$ for which 
\begin{equation}\label{eq:2DecayRate}
\lim_{m \rightarrow \infty
}\xnorm{Q_{m}e_{i}-Q_{m}e_{j}}^{1/m}=\lambda_{1}.
\end{equation}

Consider first, to fix ideas, the special case $k=2$. In this case
\eqref{eq:2DecayRate} holds with $i=1$ and $j=2$.
Assume that $n=2n'$ is even (if $n$ is odd, project onto the first
$n-1$ coordinates), and let 
\[
	x_{0}=11\dotsb 111\dotsb 1 \quad \text{and} \quad
	\tilde{x}_{0}=111\dotsb 122\dotsb 2 
\]
be the elements of $[k]^{n}$ such that $x_{0}$ has all coordinates
colored $1$, while $\tilde{x}_{0}$ has its first $n'$ colored $1$ but
its second $n'$ colored $2$. We will show that the distributions of $X_{m}$ and
$\tilde{X}_{m}$ remain at large total variation distance 
at time $m=(\theta-\alpha )\log n$.  Without loss of generality, 
assume that both of the chains $\{X_{t} \}_{t\geq 0}$ and
$\{\tilde{X}_{t} \}_{t\geq 0} $ have the same paintbox sequence
$S_{1},S_{2},\dotsc$. Then by Proposition~\ref{prop:cond indep}, the
conditional distributions of  $X_{m}$ and
$\tilde{X}_{m}$ given $\mathcal{S}=\sigma
(S_{t})_{t\geq 1}$ are product-multinomials; in particular, for any
state $x^{l}\in [k]^{[n]}$, 
\begin{align*}
	P (X^{l}_{m}=x^{l} \;\; \text{for all} \; l\in [n]\,|\, \sigma
	(S_{t})_{t\geq 1}) &=\prod_{l=1}^{n} Q_{m} (x^{l}, 1)  \quad \text{and}\\
		P (\tilde{X}^{l}_{m}=\tilde{x}^{l} \;\; \text{for all}
		\; l\in [n]\,|\, \sigma 
	(S_{t})_{t\geq 1}) &=\prod_{l=1}^{n'} Q_{m} (\tilde{x}^{l}, 1)
	\prod_{l=n'+1}^{2n'} Q_{m} (\tilde{x}^{l}, 2).
\end{align*}
But relation \eqref{eq:2DecayRate} implies that, for some $\alpha
=\alpha (\delta)>0$, if $m=(\theta-\delta)\log n$ then the
$\ell^{\infty}-$distance between the $i$th and $j$th columns of
$Q_{m}$ is at least $n^{-1/2+\alpha}$, with probability approaching 1
as $n\rightarrow\infty$. Consequently, the first $n'$ and second $n'$
coordinates of $\tilde{X}_{m}$ are (conditional on $\mathcal{S}$)
independent samples from Bernoulli distributions whose parameters
differ by at least $n^{-1/2+\alpha}$, but the $2n'$ coordinates of
$X_{m}$ are (conditional on $\mathcal{S}$) a single sample from the
same Bernoulli distribution. It follows, by
Lemma~\ref{lemma:bigBernoulliTV} (see
Remark~\ref{remark:bigBernoulliTV}, statement (B)), that the
\emph{unconditional} distributions of $X_{m}$ and $\tilde{X}_{m}$ are
at large total variation distance, because in $\tilde{X}_{m}$ the
first and second blocks of $n'$ coordinates are distinguishable
whereas in $X_{m}$ they are not.  Thus, if $m=(\theta-\delta)\log n$
then as $ n \rightarrow \infty$,
\[
	\xnorm{\mathcal{D} (X_{m})-\mathcal{D} (\tilde{X}_{m})}_{TV}
	\longrightarrow 1.
\]

The general case is proved by a similar argument. Let $n=2k (k-1)n'$
be an integer multiple of $2k (k-1)$. Break the coordinate set $[n]$
into $k (k-1)$ non-overlapping blocks of size $2n'$, one for each ordered pair
$(i,j)$ of distinct colors. In the block indexed by $(i,j)$ let
$x_{0}$ take the value $i$, and let $\tilde{x}_{0}$ take the value $i$
in the first half of the block and the value $j$ in the second
half. Let $\{X_{t} \}_{t\geq 0}$ and $\{\tilde{X}_{t} \}_{t\geq 0}$ be
versions of the EFCP with initial states $x_{0}$ and $\tilde{x}_{0}$,
respectively. Then by an argument similar to that used in the binary
case $k=2$,  if $m=(\theta-\delta)\log n$ then for large $n$, in
\emph{some} block $(i,j)$ of $\tilde{X}_{m}$ the first $n'$ and second $n'$ coordinates
of $\tilde{X}_{m}$ will be distinguishable, but in $X_{m}$ they will
not. Therefore, the unconditional distributions of $X_{m}$ and
$\tilde{X}_{m}$ will be at total variation distance near $1$.

\end{proof}

\begin{example}[\emph{Self-similar cut-and-paste chains}]\label{ex:self-similar}
Self-similar cut-and-paste chains were introduced in
\cite{Crane2011}. These are EFCP chains for which the paintbox measure
$\Sigma =\Sigma_{\nu}$ is such that if $S_{1}\sim \Sigma$ then the
columns of $S_{1}$ are i.i.d.\ with common distribution $\nu$, for some
probability distribution $\nu$ on $\simplexk$. If $S_{1},S_{2},\dotsc$
are i.i.d. with distribution $\Sigma_{\nu}$ then the random matrix
products $Q_{m}=S_{m}S_{m-1}\dotsb S_{1}$ asymptotically collapse the
simplex \emph{provided} the measure $\nu$ is nontrivial (i.e., not a
point mass), and so Theorem~\ref{theorem:ub} applies. If in addition
the measure $\nu$ has a density of class $L^{p}$ relative to Lebesgue
measure on $\simplexk$, then Theorem~\ref{theorem:cutoff} applies.
\end{example}

\subsection{Examples}\label{section:examples} We now discuss some
examples of Markov chains on $\labelednk$ whose transitions are governed
by an i.i.d.\ sequence of random partition matrices $M_1,M_2,\ldots$ with law
$\mu$, but
are not EFCP chains because $\mu$ does not coincide with $\mu_\Sigma$ for some
probability measure $\Sigma$ on $\ksimplexk$.  As a result, the examples
we show are not covered by theorems
\ref{theorem:ub} or \ref{theorem:cutoff}.  None of the examples are
EFCP chains.   We are, however, able to establish
upper bounds and, in some cases, cutoff using different techniques.
All of the chains in these examples are reversible and ergodic relative to the
uniform distribution on $[k]^{[n]}$.

\begin{example}[\emph{Ehrenfest chain on the hypercube}]\label{ex:hypercube}
For $k=2$, any $L\in\labelednk$ can be regarded as an element in $[2]^n$, or equivalently $\{0,1\}^n$.  For each $i=1,\ldots,n$ and $a\in\{1,2\}$, we define $M_{a,i}$ as the $2\times2$ partition matrix with entries
$$M_{1,i}:=\begin{pmatrix} [n]\backslash\{i\} & \emptyset\\
\{i\} & [n]	
\end{pmatrix}\quad\mbox{ or }\quad M_{2,i}:=\begin{pmatrix} [n] & \{i\}\\
\emptyset & [n]\backslash\{i\}\end{pmatrix}.$$ Let $x_0\in\labelednk$
be an initial state and first choose $a_1,a_2,\ldots$ i.i.d.\
Bernoulli(1/2) and, independently of $(a_m)$, choose $i_1,i_2,\ldots$
i.i.d.\ from the uniform distribution on $[n]$.  Then the chain
$X=(X_m,m\geq0)$ is constructed by $X_0=x_0$ and, for $m=1,2,\ldots$,
$X_{m}=M_{a_m+1,i_m}X_{m-1}$.  This corresponds to the usual Ehrenfest
chain on the hypercube, which is known to exhibit the cutoff
phenomenon at $(1/2)n\log n$; for example, see
\cite{LevinPeresWilmer}, example 18.2.2.
\end{example}

\begin{example}[\emph{General Ehrenfest chain}]\label{ex:general Ehrenfest}
A more general form of the Ehrenfest chain in the previous example is described as follows.  Fix $n\in\mathbb{N}$, take $\alpha\in(0,1)$ and choose a random subset $A\subset[n]$ uniformly among all subsets of $[n]$ with cardinality $\lfloor{\alpha n}\rfloor:=\max\{r\in\mathbb{N}:r\leq\alpha n\}$, the {\em floor} of $\alpha n$.  For $i\in[2]$ and $A\subset[n]$, we define the partition matrix $M(A,i)$ by either
\[ M(A,1):=\begin{pmatrix} [n]\backslash A & \emptyset\\ A & [n]\end{pmatrix}\quad\mbox{ or } \quad M(A,2):=\begin{pmatrix} [n] & A\\ \emptyset & [n]\backslash A\end{pmatrix}.\]
Let $A=(A_1,A_2,\ldots)$ be an i.i.d.\ sequence of uniform subsets of size $\lfloor\alpha n\rfloor$, let $I=(I_1,I_2,\ldots)$ be i.i.d.\ from the uniform distribution on $\{1,2\}$ and let $x_0\in[2]^n$.  Conditional on $A$ and $I$, we construct $X=(X_m,m\geq0)$ by putting $X_0=x_0$ and, for each $m\geq1$, 
$$X_m:=M(A_m,I_m)\cdots M(A_1,I_1)X_0.$$
We call $X$ an {\em Ehrenfest}($\alpha$) chain.

Define the coupling time $T$ by 
$$T:=\min\left\{t\geq1:\bigcup_{j=1}^t A_j=[n]\right\}.$$
Any two chains $X$ and $X'$ constructed from the same sequence $A$ will be coupled by time $T$.  

An upper bound on the distance to stationarity of the general Ehrenfest($\alpha$) chain is obtained by standard properties of the hypergeometric distribution.  In particular, let $R_t:=\#\left([n]\backslash\bigcup_{j=1}^t A_j\right)$ be the number of indices that have not appeared in one of $A_1,\ldots,A_t$.  By definition, $\{T\leq t\}=\{R_t=0\}$ and standard calculations give
\[\mathbb{P}(R_{t+1}=j|R_t=r)={{r}\choose{r-j}} {{n-r}\choose{j}}{{n}\choose{\lfloor\alpha n\rfloor}}^{-1},\mbox{ }j=0,1,\ldots,r,\]
	\[\mathbb{E}(R_t)=n\left(1-\frac{\lfloor\alpha n\rfloor}{n}\right)^t.\]
For fixed $\alpha\in(0,1)$, the $\varepsilon$-mixing time is bounded above by 
\begin{equation}\label{eq:geo-ub}
\xnorm{\mathcal{D}(X_t)-\pi}_{TV}\leq n\left(1-\frac{\lfloor\alpha n\rfloor}{n}\right)^t\leq  n \exp\{-\lfloor\alpha n\rfloor t/n\}\end{equation}
and it immediately follows, for $\beta>0$ and $t=\left(\frac{n}{2\lfloor\alpha n\rfloor}\log n+\beta\frac{n}{\lfloor\alpha n\rfloor}\right)$, that
\[\xnorm{\mathcal{D}(X_t)-\pi}_{TV}\leq n^{-1/2}\exp(-\beta)\rightarrow0\mbox{ as }\beta\rightarrow\infty.\]
When $\alpha\in(0,1/2]$, we can use proposition 7.8 from \cite{LevinPeresWilmer} and some standard theory for coupon collecting to obtain the lower bound 
\[\xnorm{\mathcal{D}(X_t)-\pi}_{TV}\geq1-8\exp\{-2\beta+1\},\]
 when $t=\left(\frac{n}{2\lfloor\alpha n\rfloor}\log n - \beta\frac{n}{\lfloor\alpha n\rfloor}\right)$.  Hence, these chains exhibit cutoff at $n/(2\lfloor\alpha n\rfloor)\log n$.  
 
 Note that the standard Ehrenfest chain (example \ref{ex:hypercube}) corresponds to $\alpha=1/n$.  
\end{example}

\begin{example}[\emph{A $\log\log n$ upper bound on mixing time}]\label{ex:log-log}
For the general Ehrenfest chains described above, the upper bound \eqref{eq:geo-ub} on mixing time can be applied more generally to sequences $\alpha:=(\alpha_1,\alpha_2,\ldots)$ in $(0,1)$.  For each $n\in\mathbb{N}$, let $\alpha_n=1-\exp\{-\log n/\log\log n\}$ and let $X^n$ be an Ehrenfest($\alpha_n$) chain.  By \eqref{eq:geo-ub}, for $t\geq(1+\beta)\log\log n$, $\beta>0$, we have
$$
\xnorm{\mathcal{D}(X^n_t)-\pi}_{TV}\leq n^{-\beta},
$$
which converges to 0 as $n\rightarrow\infty$.

In general, we can obtain an upper bound of $(1+\beta)f(n)$, where $f(n)$ is a function of $n\in\mathbb{N}$, by the relation
 $$\alpha_n=1-\exp\left\{-\frac{\log n}{f(n)}\right\}.$$
\end{example}
The space $[k]^n$ is a group under addition modulo $k$ defined by $$x+x'=x+x'-2\mbox{ (mod }k\mbox{)}+1.$$  Write $\mathbb{N}_k^n$ to denote the group $[k]^n$ together with the operation $+$, which we define by componentwise addition modulo $k$ of the coordinates of $x\in[k]^n$.  That is, for any $x,x'\in[k]^n$, we have
$$(x+x')^i=x^i+x'^i-2\mbox{ (mod }k\mbox{)} + 1.$$
This action makes the space $\labelednk$ into a group with a corresponding action $\bullet$ that can also be represented by left action of a partition matrix as follows.  If we regard $L,L'\in\labelednk$ as elements of the group $(\mathbb{N}^n_k,+)$, then we define the group action $L\bullet L'\equiv L+L'$ in the obvious way.  Alternatively, for each $L\in\labelednk$, define $\bM_L\in\matrixnk$ as the $k\times k$ matrix whose $j$th column is the $j$th cyclic shift of the classes of $L$; that is,
\begin{equation}\label{eq:partition to matrix}
\bM_L:=\begin{pmatrix} L_1 & L_k & L_{k-1} & \cdots & L_2\\
												L_2 & L_1 & L_k & \cdots & L_3\\
												L_3 & L_2 & L_1 & \cdots & L_4\\
												\vdots & \vdots & \vdots & \ddots & \vdots\\
												L_k & L_{k-1} & L_{k-2} & \cdots & L_1\end{pmatrix}.\end{equation}
Then, for every $L,L'\in\labeledNk$, we have
$$L\bullet L':=\bM_L L'.$$
\begin{example}\label{ex:group}
For $n\in\mathbb{N}$, let $\varrho_n$ be a probability measure on $\labelednk$ and let $L_0\in\labelednk$.  A $\cp_n(\varrho_n)$ chain $X$ with initial state $X_0=L_0$ can be constructed as follows.  First, generate $L_1,L_2,\ldots$ i.i.d.\ from $\varrho_n$.  Conditional on $L_1,L_2,\ldots,$ put $X_m=L_m\bullet\cdots L_{1}\bullet X_0$.  Under the definition \eqref{eq:partition to matrix} this is a cut-and-paste chain; however, the columns of each matrix are a deterministic function of one another and are not conditionally independent (as in previous examples).

Consider the case where $\varrho_n$ is a product measure of a
probability measure $\lambda$ on $[k]$ which is symmetric, i.e.\
$$\lambda(j)=\lambda(k-j+1)>0,\quad j=1,\ldots,k.$$ In this case, it
is easy to see that the $\cp_n(\varrho_n)$ chain is reversible and
hence has the uniform distribution  as its unique
stationary distribution.

For this construction of $X$, the directing measure $\mu$ on $\matrixnk$ induced by $\lambda$ is neither row-column exchangeable (RCE) nor can it be represented as $\mu_\Sigma$ for some measure $\Sigma$ on $\stochk$.  Nonetheless, the mixing time of $X$ is bounded above by $K\log n$ for some constant $K\leq2/\min_j\lambda(j)<\infty$.
\end{example}

\section{Projected cut-and-paste chains}\label{section:projected
chains}

Recall that there is a natural projection
$\Pi_n:\labelednk\rightarrow\partitionsnk$ 
from the set $\labelednk$ of labeled partitions of $[k]$ to the set
$\partitionsnk$ of unlabeled partitions. If $\{X_{m} \}_{m\geq 0}$ is a
Markov chain on the set $[k]^{n}\cong \labelednk$ whose transition
probability matrix is invariant under permutations of the labels
$[k]$, then the projection $\{\Pi_{n} (X_{m}) \}_{m\geq 0}$ is also a
Markov chain. Assume henceforth that this is the case.

Following is a simple sufficient condition for the law of an EFCP
chain to be invariant under permutations of the label set $[k]$. Say
that a probability measure $\Sigma$ on the space  $\ksimplexk$ of column-stochastic
matrices is \emph{row-column exchangeable} if the  distribution of
$S_{1}\sim \Sigma $ is invariant under permutations of the rows or the columns.

\begin{lemma}\label{lemma:rce}
If $\{X_{m} \}_{m\geq 0}$ is an EFCP chain on the set $[k]^{n}\cong
\labelednk$ whose paintbox measure $\Sigma$ is row-column exchangeable
then its  transition probability matrix is invariant under
permutations of the labels $[k]$.
\end{lemma}

\begin{proof}
Let $\gamma$ be a permutation of $[k]$ and define $\Gamma\in\matrixnk$
as the partition matrix with entries
\begin{displaymath}\Gamma_{ij}=\left\{\begin{array}{cc} [n],& \gamma(i)=j\\ \emptyset, &\mbox{otherwise.}\end{array}\right.\end{displaymath}
For $L,L'\in\labelednk$, let $P(L,L')$ denote the transition
probability from $L$ to $L'$ under the operation
\eqref{eq:cut-and-paste} with directing measure $\mu_\Sigma$.  By
row-column exchangeability of $\Sigma$, we have, for all permutations
$\gamma,\gamma'$ of $[k]$,
$$P(L,L')=P(L,\Gamma L')=P(\Gamma L,L')=P(\Gamma L,\Gamma' L')$$ for
every $L,L'\in\labelednk$.
It follows immediately that the transition probability $Q=P\Pi_n^{-1}$
of the projected chain $\Pi_n(X)$ is given by
$$Q(\Pi_n(L),\Pi_n(L'))=k^{\downarrow\#\Pi_n(L')}P(L,L'),\quad\mbox{for
every }L,L'\in\labelednk.$$
\end{proof}
Following Crane \cite{Crane2012}, we call the induced chain $\Pi:=\Pi_\infty(X)$
of an EFCP chain with RCE directing measure $\Sigma$ a {\em homogeneous 
cut-and-paste chain}.  

If the chain $\{X_{m} \}_{m\geq 0}$ is ergodic, then its unique
stationary distribution is invariant under permutations of $[k]$,
since its transition probability matrix is, and therefore projects via
$\Pi_{n}$ to a stationary distribution for the  projected chain
$\{\Pi_{n} (X_{m}) \}_{m\geq 0}$. The sufficiency principle (equation
\eqref{eq:sufficientStat}) for total variation distance (see also
Lemma~7.9 of \cite{LevinPeresWilmer}) implies
that the rate of convergence of the projected chain  $\{\Pi_{n}
(X_{m}) \}_{m\geq 0}$  is  bounded by
that of the original chain $\{X_{m} \}_{m\geq
0}$. Theorem~\ref{theorem:ub} provides a bound for this convergence
when the chain $\{X_{m} \}_{m\geq 0}$ is an EFCP.

\begin{corollary}\label{corollary:ub}
Assume that $\{X_{m}=X^{[n]}_{m} \}_{m\geq 0}$ is an EFCP chain on
$[k]^{[n]}$ whose paintbox measure $\Sigma$ is RCE and satisfies the hypothesis
of Theorem~\ref{theorem:ub} (in particular, the random matrix products
$Q_{m}$ asymptotically collapse the simplex $\simplexk$). Then for a
suitable constant $K=K_{\Sigma}<\infty$ depending only on the
distribution $\Sigma$ of $S_{1}$, and for any $\varepsilon>0$, the mixing times
$t_{\text{\textsc{mix}}}^{(n)} (\varepsilon)$ of the projected chain
$\{\Pi_{n}(X_{m}) \}_{m\geq 0}$ satisfy
 $$t_{\text{\textsc{mix}}}^{(n)}(\varepsilon)\leq K\log n$$
for all sufficiently large $n$.
\end{corollary}

\begin{theorem}\label{theorem:projected-chains}
 Suppose $\Sigma$ is a row-column exchangeable probability measure on $\stochk$.  Let $X$ be a $\cp_n(\mu_{\Sigma})$ chain and let $Y=\Pi_n(X)$ be its projection into $\partitionsnk$.  Let $t_X(\varepsilon)$ and $t_Y(\varepsilon)$ denote the $\varepsilon$-mixing times of $X$ and $Y$ respectively.  Then $$t_X(\varepsilon)=t_Y(\varepsilon).$$
 In particular, if $l(\varepsilon,n)\leq t_X(\varepsilon)\leq L(\varepsilon,n)$ are upper and lower bounds on the $\varepsilon$-mixing times of $X$, then
 $$l(\varepsilon,n)\leq t_Y(\varepsilon)\leq L(\varepsilon,n),$$
 and vice versa.
 Moreover, $X$ exhibits the cutoff phenomenon if and only if $Y$ exhibits the cutoff phenomenon.
\end{theorem}

 \begin{proof}
 If $\pi$ is the stationary distribution for $X$, then $\pi\Pi_n^{-1}$
 is the stationary distribution of $Y$.  The rest follows by the
 proceeding discussion regarding sufficiency of $\Pi_n(X)$ and the
 sufficiency principle \eqref{eq:sufficientStat}.  \end{proof}

 \begin{corollary}
Assume that the paintbox measure $\Sigma$ is row-column exchangeable
and  satisfies  hypothesis
\ref{hypothesis:L2Density}, and let $\{X_{m} \}_{m\geq 0}$ be the EFCP
on $[k]^{[n]}$ with associated paintbox measure $\Sigma$.
Then the projected
 $\cp_n(\mu_\Sigma)$ chain $\Pi_n(X)$ exhibits the cutoff phenomenon
 at time $\theta\log n$, where $\theta=-1/(2\log\lambda_1)$. 
 \end{corollary}

\bibliography{draft-refs}
\bibliographystyle{plain}
\end{document}